\documentclass[a4paper]{article}
\usepackage{amsfonts}
\usepackage{amssymb, latexsym}
\usepackage{amsmath}
\usepackage{amsthm}

\usepackage{lmodern}
\usepackage[T1]{fontenc}
\usepackage{microtype}

\usepackage[numbers]{natbib}

\usepackage{mathtools}

\usepackage{multicol}

\usepackage{xypic}

\usepackage[
colorlinks=true,       
linkcolor=blue,          
citecolor=blue,        
plainpages=false,      
pdfpagelabels,
]{hyperref}

\usepackage{graphicx}
\newcommand\twoheaduparrow{\mathop{\rotatebox[origin=c]{90}{$\twoheadrightarrow$}}}
\newcommand{\wayabove}{{\twoheaduparrow}}

\newtheorem{theorem}{Theorem}[section]
\newtheorem{proposition}[theorem]{Proposition}
\newtheorem{lemma}[theorem]{Lemma}
\newtheorem{corollary}[theorem]{Corollary}
\theoremstyle{definition}
\newtheorem{definition}[theorem]{Definition}
\theoremstyle{remark}
\newtheorem{remark}[theorem]{Remark}
\newtheorem{notation}{Notation}
\newtheorem{example}[theorem]{Example}
\numberwithin{equation}{section}

\usepackage[shortlabels]{enumitem}
\setlist[enumerate]{leftmargin=*,align=left,labelindent=\parindent}

\makeatletter
\newcommand{\myitem}[1][]{%
\item[#1]\protected@edef\@currentlabel{#1}\ignorespaces%
}
\makeatother

\DeclareMathOperator{\id}{\mathrm{id}}
\newcommand{\cov}{\mathrel{\vartriangleleft}}
\newcommand{\amp}{\mathrel{\&}}

\DeclareMathOperator{\defeqiv}{\stackrel{\mathrm{def}}{\iff}}
\DeclareMathOperator{\defeql}{\overset{\mathrm{def}}{\,=\mathrel{\mkern-9mu}=\,}}
\DeclareMathOperator*{\medwedge}{{\textstyle{\bigwedge}} }
\DeclareMathOperator*{\medvee}{{\textstyle{\bigvee}}}
\DeclareMathOperator{\wb}{\mathsf{wb}}

\newcommand{\Fin}[1]{\mathsf{Fin}(#1)}
\newcommand{\PFin}[1]{\mathsf{Fin}^{+}(#1)}

\newcommand{\Pow}[1]{\mathcal{P}(#1)}
\newcommand{\Sat}[1]{\mathrm{Sat}(#1)}

\newcommand{\sat}{\mathcal{A}}
\newcommand{\Model}[1]{\mathrm{Mod}(#1)}
\newcommand{\Viet}{\mathcal{V}}
\newcommand{\VViet}{\mathfrak{V}}
\newcommand{\One}{ \left\{ * \right\}}
\newcommand{\Terminal}{\mathbf{1}}
\newcommand{\Nat}{\mathbb{N}}
\newcommand{\Rat}{\mathbb{Q}}
\newcommand{\Bin}{{ \left\{ 0,1 \right\}}}
\newcommand{\FR}{\mathcal{R}}
\newcommand{\FTop}{\textbf{\textrm{\textup{FTop}}}}
\newcommand{\FTopi}{\textbf{\textrm{\textup{FTop$_{i}$}}}}
\newcommand{\BCov}{\textbf{\textrm{\textup{BCov}}}}
\newcommand{\BCovi}{\textbf{\textrm{\textup{BCov$_{i}$}}}}
\newcommand{\KReg}{\textbf{\textrm{\textup{KReg}}}}
\newcommand{\LFCov}{\textbf{\textrm{\textup{ContCov}}}}
\newcommand{\meets}{\between}

\newcommand{\op}{\mathrm{op}}

\newcommand{\Patch}{\mathrm{Patch}}
\newcommand{\Perf}{\mathrm{PSub}}
\newcommand{\SubTop}{\mathrm{SubTop}}
\newcommand{\Located}{\mathcal{L}}
\newcommand{\LLocated}{\mathfrak{L}}
\newcommand{\Scott}{\Sigma}

\DeclarePairedDelimiter{\abs}{\textup{`}}{\textup{'}}
\newcommand{\Closed}[1]{\abs{#1}}
\newcommand{\KFit}[1]{\kappa_{#1}}
\newcommand{\rcut}{\textup{\textbf{\textsf{r}}}}
\newcommand{\lcut}{\textup{\textbf{\textsf{l}}}}
\newcommand{\sq}{\Box}
\newcommand{\dia}{\Diamond}

\newcommand{\Lower}[1]{\mathrm{P_L}(#1)}

\newcommand{\elPT}[1]{\mathbb{#1}}

\newcommand{\CZF}{\textmd{\textrm{\textup{CZF}}}}
\title{Geometric theories of patch and \\ Lawson topologies}

\author{Tatsuji Kawai\\[.5em]
\normalsize Dipartimento di Matematica,
\normalsize Universit\`{a} di Padova\\
\normalsize via Trieste 63, 35121 Padova, Italy\\
\normalsize\texttt{tatsuji.kawai@math.unipd.it}}

\date{\today}

\begin{document}
\maketitle
\begin{abstract}
  We give geometric characterisations of patch and Lawson
  topologies in the context of predicative point-free topology
  using the constructive notion of located subset. We present
  the patch topology of a stably locally compact formal topology
  by a geometric theory whose models are the points of the
  given topology that are located, and the Lawson topology of a
  continuous lattice by a geometric theory whose models
  are the located subsets of the given lattice.
  We also give a predicative presentation of the frame of perfect
  nuclei on a stably locally compact formal topology, and show that it
  is essentially the same as our geometric presentation of the patch
  topology.  Moreover, the construction of Lawson topologies
  naturally induces a monad on the category of compact regular formal
  topologies, which is shown to be isomorphic to the Vietoris monad.


\end{abstract}

\section{Introduction}\label{sec:Introduction}
The purpose of this paper is to give yet another constructions of patch and
Lawson topologies in the context of predicative point-free topology.  A point-free
approach to patch topologies (and in particular Lawson topologies of 
domains) was initiated by \citet{escardo2001regular}, who constructed
the patch topology of a stably locally compact locale as the frame
of perfect nuclei. His construction is constructive in the sense
of topos theory but impredicative.
Later, \citet{CoquandZhangPrediativePatch} gave a predicative
construction of the patch topology of a stably compact (i.e.\ stably
locally compact and compact) locale using the notion of entailment relation
\cite{cederquist2000entailment}. These authors are motivated by the
special case of spectral locales, in which the patch construction
can be seen as a solution to the problem of adding the Boolean
complements to a distributive lattice (cf.\ Section
\ref{sec:Spectral}); in a more general case, they consider the patch
construction to be a solution to a certain generalisation of this
universal problem. 

In this paper, we give another account of patch and Lawson
topologies in the context of formal topology
\cite{Sambin:intuitionistic_formal_space}, motivated by the
constructive notion of located subset. A subset $A \subseteq X$ of a
metric space $(X,d)$ is \emph{located} if for any point $x \in X$, the
distance
  $
  d(x,A) \defeql \inf \left\{ d(x,y) \mid y \in A\right\}
  $
exists as a Dedekind real, i.e.\ for any positive rational numbers
that are apart $p < q$, either there exists $y \in A$ such that
$d(x,y) < q$ or there is no $y \in A$ such that $d(x,y) < p$.
Constructively, locatedness is a non-trivial notion and plays a very
important role in the development of constructive mathematics
\cite{Bishop-67,Intuitionism}. \citet{Spitters10LocatedOvert} gave a
point-free definition of located subset in the context of compact
regular formal topology, generalising the above metric notion to the
point-free setting. He defined a subset $V$ of the base of the compact
regular formal topology to be located if it is a completely prime
upper set such that
\[
  a \ll b \implies a \notin V \vee b \in V
\]
for all basic elements $a,b$, where $\ll$ is the way-below relation.
Since this notion only depends on the structure of the continuous
lattice of a compact regular formal topology, it also makes good sense
in a more general context of continuous lattice (see Section
\ref{sec:Located}).

Then, the main result of this paper is the following: the patch
topology of a stably locally compact formal topology is presented by
a geometric theory whose models are the formal points of the given
topology that are located, and the Lawson topology of a continuous
lattice is presented by a geometric theory whose models are the
located subsets of the given continuous lattice; see Section
\ref{sec:Patch}, and Section \ref{sec:LawsonTop}.
The crucial step is to observe that the above notion
of located subset is not geometric but there is an equivalent
geometric characterisation of a located subset which is analogous
to the Dedekind cuts; see Section \ref{sec:Located}.

Informally, we can say that the patch topology is the space of located
points and the Lawson topology is the space of located subsets. We
have thus given spatial characterisations of patch and Lawson
topologies in the constructive terms. However, we expect that these
characterisations are also relevant to the classical theory of effective
domains; specifically, it would be interesting to see a precise
connection between the notion of located subset and that of Lawson
computable element of a domain discussed by
\citet{EscardoCompContentLawsonTop}.
Our predicative presentations of patch and Lawson topologies
should make this comparison easier, since the basic notions in
the effective domain theory are defined in terms of the base of a
topology.

Apart from the above mentioned results, we give a predicative
presentation of the frame of perfect nuclei on a stably locally
compact locale in terms of formal topology, and show that this
presentation is essentially the same as the geometric presentation of
the patch topology; see Section \ref{sec:PerfectNuclei}.  This makes
precise the connection between our work and the previous work by
\citet{escardo2001regular} in the topos theoretic context.  Moreover,
the construction of Lawson topologies determines the right adjoint to
the forgetful functor from the category $\KReg$ of compact regular
formal topologies to that of continuous lattices and perfect maps.
Hence, the adjunction induces a monad on $\KReg$, which is shown to be
isomorphic to the Vietoris monad on $\KReg$.  Indeed, the connection
between Lawson topologies of continuous lattices and the Vietoris
monad on the category of compact Hausdorff spaces is well known
\cite{gierz2003continuous,johnstone-82};
however, our result seems to be the first constructive and point-free
account of this connection; see Section~\ref{sec:Vietoris}.

Throughout this paper, we work constructively (i.e.\ using
intuitionistic logic) but also predicatively in a generalised sense.
This means that we do not admit the notion of powerset, but allow
certain forms of generalised inductive definitions.  Specifically, our
results can be carried out in the constructive set theory \CZF{} with the
Regular Extension Axiom; see \citet{Aczel-Rathjen-Note} for the
details of \CZF.
Our basic reference for the classical domain theory is
\citet{gierz2003continuous}.

\begin{notation}
We write $\left\{ * \right\}$ for a fixed one point set.
If $S$ is a set, $\Pow{S}$ denotes the class of subsets of $S$.
Predicatively, $\Pow{S}$ cannot be a set
unless  $S = \emptyset$.
A set $A$
is \emph{finitely enumerable} if there exists a surjection $f \colon
\left\{0,\dots,n-1 \right\} \to A$ for some $n \in \Nat$.
The \emph{set} of finitely enumerable
subsets of a set $S$ is denoted by $\Fin{S}$.
For subsets
$U,V \subseteq S$, write $U \meets V$ if the intersection $U \cap
V$ is inhabited.  If $r$ is a relation between sets $X$ and $S$, its
inverse image map $r^{-} \colon \Pow{S} \to \Pow{X}$ is defined by
  $
  r^{-} U \defeql \left\{ x \in X \mid \left( \exists a \in U\right) x
  \mathrel{r} a \right\}.
  $
We often write $r^{-}a$ for $r^{-}\left\{ a \right\}$.
\end{notation}

\section{Formal topologies}\label{sec:FormalTopology}
We give some backgrounds on basic covers and formal
topologies, the predicative notions of complete suplattices and
frames. Our exposition is based on  \citet{ConvFTop} and
\citet{Fox05}. The knowledgeable reader is advised to skip this section.

\subsection{Basic covers}\label{sec:BasicCover}
\begin{definition}
  \label{def:BCov}
  A \emph{basic cover} is a pair $\mathcal{S}=(S, \cov)$
  of a set $S$ and a relation $\cov$ between $S$ and $\Pow{S}$ such that 
  \[
    \sat U \defeql \left\{ a \in S \mid a \cov U \right\}
  \]
  is a set for each $U \subseteq S$ and that
  \begin{enumerate}
    \item $a \in U \implies a \cov U$,
    \item $a \cov U \amp U \cov V \implies a \cov V$,
  \end{enumerate}
  where
     $U \cov V \defeqiv \left( \forall a \in U \right) a \cov V$.
  The set $S$ is called the \emph{base} of $\mathcal{S}$,
  and the relation $\cov$ is called a \emph{cover} on $S$,
  or the cover of $\mathcal{S}$.  
  Given a subset $U \subseteq S$, a subset $\sat U$ is called a
  \emph{saturation} of $U$, and $U$ is said to be \emph{saturated} if
  $U = \sat U$.  The collection $\Sat{\mathcal{S}} = \left\{ \sat U
    \mid U \in \Pow{S} \right\}$ of saturated subsets of
    $\mathcal{S}$ can be identified with $\Pow{S}$ together with the
    equality
  \[
    U =_{\mathcal{S}} V \defeqiv \sat U =  \sat V.
  \]
$\Sat{\mathcal{S}}$ 
forms a complete suplattice with joins defined by $\bigvee_{i
\in I}\sat U_i = \sat \bigcup_{i\in I} U_i$ 
for any set-indexed family $(U_i)_{i \in I}$ of subsets of $S$.
\end{definition}
\begin{notation}
  We use letters $\mathcal{S},\mathcal{S}',\dots$ to denote basic
  covers,
  and use $S,S',\dots$, $\cov,\cov',\dots$, and $\sat,
  \sat',\dots$ to denote their bases, covers, and saturation
  operations.
\end{notation}

\begin{definition}\label{def:BCM}
  Let $\mathcal{S}$ and $\mathcal{S}'$ be basic covers.
  A relation $r \subseteq S \times S'$ is called a \emph{basic cover
  map} from $\mathcal{S}$ to $\mathcal{S}'$ if 
  \begin{equation}\label{FTM3}
    a \cov' U \implies r^{-} a  \cov r^{-}U
  \end{equation}
  for each $a \in S'$ and $U \subseteq S'$. 
  Basic cover maps from $\mathcal{S}$ and $\mathcal{S}'$
  are ordered by
  \[
    r \leq s \defeqiv \left( \forall a \in S \right) r^{-} a \cov
    s^{-}a,
  \]
  and they are defined to be \emph{equal}  if $r \leq
  s \amp s \leq r$. 

  The basic covers and basic cover maps form a category $\BCov$. 
  The identity map on a basic cover is the identity relation on its
  base, and the composition
  of basic cover maps $r \colon \mathcal{S} \to \mathcal{S}'$ and
  $s \colon \mathcal{S}' \to \mathcal{S}''$ is given by the
  composition of their underlying relations.
\end{definition}
A basic cover map $r \colon \mathcal{S} \to \mathcal{S}'$ determines
a suplattice homomorphism
$f_{r} \colon \Sat{\mathcal{S}'}
\to \Sat{\mathcal{S}}$ by
\begin{equation}\label{def:FuncBCovSupLat}
  f_{r}(\sat' U) = \sat r^{-}U.
\end{equation}
The assignment $r \mapsto f_{r}$ is a contravariant functor
from $\BCov$ to the category of set-based complete
suplattices, where a complete suplattice $(X, \bigvee)$ is 
\emph{set-based} if it is equipped with a subset $S \subseteq X$
such that $S_{x} \defeql \left\{ a \in S \mid a \leq x \right\}$ is a set for
each $x \in X$ and $\bigvee S_{x} = x$.
This functor forms a part of the dual equivalence between the two
categories; see \citet[Proposition 2.4]{ConvFTop}.

The terminal object of $\BCov$ is defined by
$\Terminal \defeql \left(\left\{ * \right\}, \in \right)$,
which corresponds to the truth values $\Pow{\One}$.
A global point $\Terminal \to \mathcal{S}$ of a basic cover
$\mathcal{S}$ can be characterised as follows.
\begin{definition}
  Let $\mathcal{S}$ be a basic cover.
  A subset $V \subseteq S$ is \emph{splitting} if
  \[
    a \in V \amp  a \cov U \implies V \meets U
  \]
  for all $a \in S$ and $U \subseteq S$.
\end{definition}
A splitting subset $V \subseteq S$ bijectively corresponds to a basic
cover map $r_{V} \colon \Terminal \to \mathcal{S}$ given by
  $
  * \mathrel{r_V} a \defeqiv  a \in V.
  $
Conversely, a basic cover map $r \colon \Terminal \to \mathcal{S}$
determines a splitting subset by
  $
  V_{r} \defeql \left\{ a \in S \mid * \mathrel{r} a \right\}.
  $

\subsection{Formal topologies}
\begin{definition}
  A \emph{formal topology} is a triple $\mathcal{S}=(S, \cov, \leq)$
  where $(S, \cov)$ is a basic cover and $(S, \leq)$ is a preorder
  such that 
  \begin{enumerate}
    \item $a \leq b \implies a \cov b$,
    \item $a \cov U \amp a \cov V \implies a \cov  U \downarrow V$
  \end{enumerate}
  for all $a,b \in S$ and $U,V \subseteq S$, where
  \[
    U \downarrow V \defeql \left\{ c \in S \mid \left( \exists a \in U
      \right) \left( \exists b \in V \right) c \leq a \amp c \leq
    b\right\}.
  \]
  We write $a \downarrow b$ for
  $\left\{ a \right\} \downarrow \left\{ b \right\}$
  and $a \downarrow U$ for $\left\{ a \right\} \downarrow U$.
\end{definition}
The collection $\Sat{\mathcal{S}}$ of saturated subsets of a formal
topology $\mathcal{S}$ has finite meets
with top $S$ and meets $\sat U \wedge \sat V = \sat(U \downarrow V)$
which distribute over the join structure on $\Sat{\mathcal{S}}$
as given in Definition \ref{def:BCov}. Hence, $\Sat{\mathcal{S}}$
forms a frame.

\begin{definition}\label{def:FTM}
  Let $\mathcal{S}$ and $\mathcal{S}'$ be formal topologies.
  A \emph{formal topology map} from $\mathcal{S}$ to $\mathcal{S}'$
  is a basic cover map $r \colon \mathcal{S} \to \mathcal{S}'$
  such that 
  \begin{enumerate}[({FTM}1)]
    \item\label{FTM1} $S \cov r^{-}S'$,
    \item\label{FTM2} $r^{-} a \downarrow r^{-} b \cov r^{-}(a \downarrow' b)$
  \end{enumerate}
  for all $a,b \in S'$.
\end{definition}
The formal topologies and formal topology maps 
form a subcategory $\FTop$ of $\BCov$. The assignment given by
\eqref{def:FuncBCovSupLat} restricts to the dual equivalence between
$\FTop$ and the category of set-based frames.

\begin{definition}\label{def:Pt}
  Let $\mathcal{S}$ be a formal topology. A splitting subset $\alpha
  \subseteq S$ is called a \emph{formal point} of $\mathcal{S}$ if
  $\alpha$ is inhabited and 
  $a, b \in \alpha \implies \alpha \meets (a \downarrow b)$
  for all $a, b \in S$.
\end{definition}
The formal points of a formal topology bijectively
correspond to the formal topology maps from the terminal object
$\Terminal$.
The correspondence is the restriction of the one between the splitting
subsets and the basic cover maps from $\Terminal$.

\subsection{Inductively generated formal topologies}
The notion of inductively generated formal topology by \citet[Theorem
3.1]{Coquand200371} allows us to define a formal topology by a set of
axioms.
\begin{definition}
An \emph{axiom-set} on a set $S$
is a pair $(I,C)$ where $(I(a))_{a \in S}$ is a family of sets
indexed by $S$, and $C$ is a family $(C(a,i))_{a \in S, i \in I(a)}$
of subsets  of $S$.
\end{definition}
Given an axiom-set $(I,C)$ on a set $S$, we can 
inductively define a cover $\cov_{I,C}$ on $S$ by the following rules:
  \begin{gather*}
    \frac{a \in U}{a \cov_{I,C} U} \,\text{(reflex)},
    \qquad
    \frac{i \in I(a) \quad
    C(a,i) \cov_{I,C} U}{a \cov_{I,C} U} \, \text{(infinity)}.
  \end{gather*}
A basic cover $\mathcal{S}=(S, \cov)$ is \emph{inductively
generated} if there is an axiom-set $(I,C)$ on $S$ so
that the cover $\cov$ is inductively defined by the above rules.

In order to generate a cover of a formal topology, we start from an
axiom-set on a preordered set which is localised with respect to the
given preorder.
\begin{definition}
  Let $(S,\leq)$ be a preordered set and 
  let $(I,C)$ be an axiom-set on $S$.  We say that $(I,C)$ is
  \emph{localised} with respect to $\leq$ if for any $a,b \in S$,
\begin{align*}
  a \leq b \implies
  \left( \forall i \in I(b) \right)\left( \exists j \in I(a) \right)
  C(a,j) \subseteq a \downarrow  C(b,i).
\end{align*}
\end{definition}
\begin{remark}
Given an axiom-set $(I,C)$ on $S$ and a preorder $\leq$ on $S$, we can
always localise $(I,C)$ with respect to $\leq$ by defining a new
axiom-set $(I',C')$, the \emph{localisation of $(I,C)$}, as follows:
\begin{align*}
  I'(a) &\defeql \sum_{a \leq b}I(b), & C'( (b,i), a) &\defeql a \downarrow C(i,b).
\end{align*}
\end{remark}

If $(I,C)$ is a localised axiom-set with respect to a preorder $(S
,\leq)$, then we can define a formal topology $\mathcal{S} = (S,
\cov_{I,C}, \leq)$, where the cover $\cov_{I,C}$ is inductively defined by
the following rules:
\begin{gather*}
  \frac{a \in U}{a \cov_{I,C} U} \, \text{(reflex)},
  \quad
  \frac{a \leq b \cov_{I,C} U}{a \cov_{I,C} U} \,\text{($\leq$)}, 
  \quad
  \frac{i \in I(a) \quad
  C(a,i) \cov_{I,C} U}{a \cov_{I,C} U} \,\text{(infinity)}.
\end{gather*}
A formal topology $\mathcal{S}=(S, \cov, \leq)$ is
\emph{inductively generated} if there is a localised
axiom-set $(I,C)$ on $S$ so that the cover $\cov$ is inductively defined
by the above three rules.
We also say that a formal topology $(S, \cov, \leq)$ is
inductively generated by an (not necessarily localised) axiom-set
$(I,C)$ if the cover $\cov$ is generated by the localisation of
$(I,C)$ with respect to $\leq$.

\subsection{Geometric theories}\label{sec:GeometricTheory}
A geometric theory is a predicative
presentation of a frame (and hence a formal topology) by generators and
relations \cite{vickers1989topology}.
A geometric theory allows for a spatial understanding
of the point-free space presented by the theory by considering
what an arbitrary model of the theory is.

\begin{definition}
  Given a set $P$ of \emph{propositional symbols} (or
  \emph{generators}), a \emph{geometric theory} over $P$ is a set of
  \emph{axioms} of the  form
  \[
    \medwedge A \vdash \bigvee_{i \in I} \medwedge B_i
  \]
  where $I$ is a set and $A, B_i$ are finitely enumerable subsets of 
  $P$. We write $\top$ for $\medwedge \emptyset$ and $\bot$ for $\bigvee
  \emptyset$. Conjunction and disjunction symbols will be omitted if
  they are single conjunct or disjunct (e.g.\ we write $a$ instead of
  $\bigvee \left\{ \medwedge \left\{ a \right\} \right\}$).
\end{definition}

\begin{definition}
Given a geometric theory $T$ over $P$, define a localised axiom-set
$(I_{T},C_{T})$ on $\Fin{P}$ with respect to the reverse inclusion 
order on $\Fin{P}$ whose axioms consist of a pair
\[
  C \cup A \cov_{I_T,C_T} \left\{ C \cup B_i \mid i \in I \right\}
\]
for each $\medwedge A \vdash \bigvee_{i \in I}
\medwedge B_i \in T$ and $C \in \Fin{S}$.
We write $\mathcal{S}_{T} \defeql (\Fin{P}, \cov_{T}, \supseteq)$ for the
topology generated by $(I_{T},C_{T})$, and call
$\mathcal{S}_{T}$ \emph{the formal topology presented by $T$}.
\end{definition}

The frame $\Sat{\mathcal{S}_{T}}$ of saturated subsets of
$\mathcal{S}_{T}$ is the frame presented by $T$:
there is a function
$i_{T} \colon P \to \Sat{\mathcal{S}_{T}}$ given by
  $
  i_{T}(a) = \sat\left\{ \{a\} \right\}
  $
which respects the axioms of $T$, i.e.\
  $
  \medwedge_{a \in A} i_{T}(a) \leq \bigvee_{i \in I} \medwedge_{b
  \in B_i} i_{T}(b)
  $
  whenever
$\medwedge A \vdash \bigvee_{i \in I} \medwedge B_i \in T$,
and for any frame $X$ and a function $f \colon P \to X$ that
respects the axioms of $T$, there exists a unique
frame homomorphism $F \colon \Sat{\mathcal{S}_{T}} \to X$ such that $F
\circ i_{T} = f$, which is given by 
\[
  F(\sat_T{\mathcal{U}}) \defeql \bigvee_{A \in \mathcal{U}}
  \medwedge_{a \in A} f(a)
\]
for each $\mathcal{U} \subseteq \Fin{P}$. 

In particular, in order to define a formal topology map
$r \colon \mathcal{S} \to \mathcal{S}_{T}$ from an arbitrary formal
topology $\mathcal{S}$, it suffices
to define a relation $r$ between the base of $\mathcal{S}$
and the generators of $T$ and then verify that
\begin{align}
\label{eq:PreservAx}
  r^{-}\left\{ a_{0} \right\} \downarrow \cdots \downarrow
  r^{-}\left\{ a_{n-1} \right\} \cov \bigcup_{i
  \in I} r^{-} \left\{ b^{i}_{0} \right\} \downarrow \cdots \downarrow
  r^{-}\left\{ b^{i}_{n_{i}-1} \right\}
\end{align}
  for each axiom 
  $
  a_{0} \wedge \cdots \wedge a_{n-1}
  \vdash
  \bigvee_{i \in I} b^{i}_{0} \wedge \cdots \wedge
  b^{i}_{n_{i}-1} \in T
  $.
  Then, such a relation extends to a formal topology map 
  $r \colon \mathcal{S} \to \mathcal{S}_{T}$ by $b \mathrel{r} A
  \defeqiv
  \left( \forall a \in A \right) b \cov r^{-}a$.

\begin{definition}
  Let $T$ be a geometric theory over $P$. A \emph{model} of $T$ is
  a subset $m \subseteq P$ such that
  \[
    A \subseteq m \implies \left( \exists i \in I \right) B_i \subseteq
    m
  \]
  for all $\medwedge A \vdash \bigvee_{i \in I} \medwedge B_i \in T$.
  The class of models of $T$ is denoted by $\Model{T}$.
\end{definition}
A model $m$ of $T$ is equivalent to a function $f_{m} \colon
P \to \Pow{\One}$ which respects all the axioms of
$T$, and thus to a frame homomorphism
$F_{m} \colon \Sat{\mathcal{S}_{T}} \to \Pow{\One}$, i.e.\ to a formal point of
$\mathcal{S}_{T}$.

\section{Located subsets}\label{sec:Located}
As we noted in Introduction \ref{sec:Introduction},  our
motivation of patch and Lawson topologies comes from the
point-free notion of located subset introduced by
\citet{Spitters10LocatedOvert}. We show
that this point-free notion makes good sense in a more general setting
of continuous basic cover, a predicative notion of continuous
lattice.

\subsection{Continuous basic covers}
Predicative notions of continuous lattice and locally compact locale
were given by \citet{negri1998continuous}, where they are called
locally Stone infinitary preorder and locally Stone formal topology
respectively. Here, we review those notions by the names of 
continuous basic cover and locally compact formal topology.

\begin{definition} \label{def:LK}
Let $\mathcal{S}$ be a basic cover. For each $a, b \in S$ define
\begin{equation}
   a \ll b \defeqiv \bigl( \forall U \in \Pow{S} \bigr)\left[ \:
   b \cov U
   \implies \bigl( \exists U_0 \in \Fin{U} \bigr)\; a \cov U_0 \right].
\end{equation}
We say that  $a$ is \emph{way-below} $b$ if
$a \ll b$.
We extend $\ll$ to the subsets of $S$ by
\begin{equation*}
   U \ll V \defeqiv \bigl( \forall W \in \Pow{S} \bigr)\left[
   V \cov W
   \implies \bigl( \exists W_0 \in \Fin{W} \bigr)\; U \cov W_0 \right].
\end{equation*}
For $a \in S$ and $U \subseteq S$, we write $a \ll U$ for
$ \left\{ a \right\} \ll U$ and $U \ll a$ for $U \ll \left\{ a \right\}$. 

A basic cover $\mathcal{S}$ is \emph{continuous} if it is
equipped with a function $\wb \colon S \to \Pow{S}$ such that
\begin{enumerate}
  \item\label{eq:wb1} $\bigl( \forall b \in \wb(a) \bigr)\: b \ll a$, 
  \item\label{eq:wb2} $a \cov \wb(a)$
\end{enumerate}
for all $a \in S$. A formal topology $\mathcal{S}$ is \emph{locally
compact} if it is continuous as a basic cover.%
\end{definition}
  Since the relation $\ll$ is a proper class in general, the
  function $\wb \colon S \to \Pow{S}$ in Definition
  \ref{def:LK} is indispensable for a predicative definition.
  Note that if $\mathcal{S}$ is continuous
  with an associated function $\wb \colon S \to \Pow{S}$, then
   \[
   a \ll b \iff \bigl( \exists A \in \Fin{\wb(b)} \bigr)\; a \cov A
   \]
   for all $a,b \in S$. Hence, if such a function $\wb$ exists, 
   the relation $\ll$ is a set.
\begin{remark}\label{rem:LKInd}
  It is easy to see that every continuous basic cover
  $(S, \cov)$ is inductively
  generated by an axiom-set $(I,C)$ given by
  $I(a) \defeql \One + \left\{ A \in \Fin{S} \mid 
    a \cov A \right\}$,
  $C(a,*) \defeql \wb(a)$ and $C(a,A) \defeql A$.
\end{remark}
If $\mathcal{S}$ is a continuous basic cover,
we have $\sat U = \sat \bigcup \left\{ B \in \Fin{S} \mid B \ll U \right\}$
for each $U \subseteq S$. Thus, the suplattice
$\Sat{\mathcal{S}}$ is a continuous lattice
with the base $\left\{ \sat B \mid B \in \Fin{S} \right\}$.
Then, the Scott topology $\Scott(\mathcal{S})$ of $\Sat{\mathcal{S}}$
can be defined as a formal topology $(\Fin{S}, \cov_{\Scott},
\leq_{\Scott})$ where
\begin{align}
  \begin{aligned}
  \label{def:Scott}
  A \leq_{\Scott} B &\defeqiv B \subseteq A, \\
  A \cov_{\Scott} \mathcal{U} &\defeqiv \wayabove A \subseteq
  \bigcup_{B \in \mathcal{U}}\! \wayabove B,\\
  \wayabove A &\;\:\defeql \left\{ B \in \Fin{S} \mid A \ll B \right\}.  \end{aligned}
\end{align}

The interpolation property of $\ll$ is well known.
\begin{lemma}
  \label{lem:LKInterpolate}
  Let $\mathcal{S}$ be a continuous basic cover, and let
  $U,V \subseteq S$. Then,
  \[
    U \ll V \implies \left( \exists A \in \Fin{S} \right)
    U \ll A \amp \left( \forall a \in A \right) \left( \exists b \in
    V \right) a \ll b.
  \]
\end{lemma}
A natural notion of morphism between continuous basic covers is a perfect map.
\begin{definition}
  Let $\mathcal{S}$ and $\mathcal{S}'$ be continuous basic covers.
  A basic cover map $r \colon \mathcal{S} \to \mathcal{S}'$ is \emph{perfect} if 
  \[
    a \ll' b \implies r^{-}a \ll r^{-} b
  \]
  for all $a,b \in S'$.
  A \emph{perfect} formal topology map between locally compact formal
  topologies is defined similarly.
\end{definition}
The continuous basic covers and perfect maps form a category $\LFCov$.

\subsection{Located subsets}
\begin{definition}\label{def:located}
  Let $\mathcal{S}$ be a continuous basic cover. A subset $V
  \subseteq S$ is \emph{located} if it is a splitting subset of
  $\mathcal{S}$, and moreover satisfies 
  \[
  a \ll b \implies a \notin V \vee b \in V
  \]
  for all $a,b \in S$.
  A formal point of a locally compact formal topology
  is said to be \emph{located} if it is a located subset.
\end{definition}
\begin{remark}
  We will see that every located subset of a continuous basic
  cover can be seen as a located point of some suitable topology
  (cf.\ Theorem \ref{thm:LawsonTop}).
\end{remark}
Classically, every splitting subset of a continuous basic cover is located;
constructively this is not a case, which is clear from the
following examples.
\begin{example}
Let $\mathcal{P}\omega$ be a basic cover $(\Fin{\Nat}, \cov_{\omega})$ where
\begin{align*}
  A \cov_{\omega} U &\defeqiv \left( \exists B \in U \right) B
  \subseteq A.
\end{align*}
$\mathcal{P}\omega$ is continuous since $A \ll A$ for all $A \in
\Fin{\Nat}$.
Then, a subset $V \subseteq \Fin{\Nat}$ is splitting if and
only if $V$ is closed downwards with respect to $\subseteq$,
and a splitting subset $V$ is located if and only if it is
detachable, i.e.\ $A \in V \vee A \notin V$  for each $A \in \Fin{\Nat}$.
Furthermore, it is easy to see that a formal point $\alpha$ of
$\mathcal{P}\omega$ bijectively corresponds to a subset $\bigcup
\alpha$ of $\Nat$, and that a located point bijectively corresponds to a
detachable subset of $\Nat$. 
\end{example}

\begin{example}\label{eg:Dedekind}
  Let $\FR^{u}$ be a basic cover $(\Rat, \cov_{u})$ over the rationals 
  $\Rat$, where
  \[
    q \cov_{u} U \defeqiv \left( \forall p < q\right) \left( \exists q'
    \in U \right) p < q'.
  \]
  The basic cover $\FR^{u}$ is continuous with the function $\wb(q) = \left\{ p
  \in \Rat \mid p < q \right\}$.
  Then, a subset $V \subseteq \Rat$ is splitting if and only if it is
  an upper real, i.e.
      \[
        q \in V \iff \left( \exists p < q \right) p \in V,
      \]
  and a splitting subset is a formal point if and only if it is
  inhabited.
  Moreover, a splitting subset $V$ is located if and only if it is
  an extended (i.e.\ non-finite) Dedekind real, i.e.\
  \[
    p < q \implies p \notin V \vee q \in V.
  \]
  Thus, a located point is an extended Dedekind real with an upper bound.

  It is well known that an extended Dedekind real can be equivalently
  defined as an extended Dedekind cut, namely a pair $(L,U)$ of
  subsets of $\Rat$ such that
  \begin{enumerate}
    \item $ q \in U \iff \left( \exists q' < q \right) q' \in U,$
    \item $ p \in L \iff \left( \exists p' > p \right) p' \in L,$
    \item $ L \cap U = \emptyset$,
    \item $ p < q \implies p \in L \vee q \in U$.
  \end{enumerate}
 See \citet[Chapter 5, Exercise 5.5.3]{ConstMathI}. Although extended
 Dedekind reals (i.e.\ right cuts) and extended Dedekind cuts are
 equivalent, the latter can be characterised as the models of a
 certain geometric theory (see e.g.\ \citet[Section
 2]{LocalicCompletionGenMet}), while this is not the case for extended
 Dedekind reals. To characterise an extended Dedekind
 real as a model of a certain propositional theory, we need an
 implication symbol, which is implicit in the negation $p \notin V$.
\end{example}
\begin{example}
Formal Cantor space $\mathcal{C}$ is a basic cover $(\Bin^{*},
\cov_{\mathcal{C}})$ over the finite binary sequences $\Bin^{*}$ such
that
\begin{align*}
  a \cov_{\mathcal{C}} U &\defeqiv \left( \exists k \in \Nat \right)
  \left( \forall c \in a[ k] \right) \left( \exists b \in U \right) b \preccurlyeq c
\end{align*}
where $a[k] \defeql \left\{ a * b \mid |b| = k \right\}$. 
Here, $a \preccurlyeq b$ means that $a$ is an initial segment of $b$, and
$a * b$ and $|a|$ denote the concatenation and the length of sequences
respectively. Note that $a \cov_{\mathcal{C}} U$ if an only if $U$ is
a uniform bar over $a$. A subset $V \subseteq \Bin^{*}$ is splitting
if and only if
  $
  a \in V \iff \left( \exists i \in \Bin \right) a * \langle i \rangle
  \in V,
  $
and a splitting subset is located if and only if it is detachable.
Hence, the  located subsets of $\mathcal{C}$ are exactly the sub-spreads of the binary
spread; see \citet[Section 3.1]{Intuitionism}. It is also easy to
see that every formal point of $\mathcal{C}$ is located.

The interested reader is referred to \citet{Spitters10LocatedOvert}
and \citet{PointFreeBishopCMSpa} in which the
point-free notion of located subset is related to the metric notion.
\end{example}

A located subset of a continuous basic cover $\mathcal{S}$
admits a natural characterisation as a global point of $\mathcal{S}$ in
the category $\LFCov$.
\begin{proposition}
  The located subsets of a continuous basic cover $\mathcal{S}$
  bijectively correspond to the perfect maps from $\Terminal$ to
  $\mathcal{S}$.
\end{proposition}
\begin{proof}
  It suffices to show that a splitting subset $V$ of
  $\mathcal{S}$ is located if and only if the morphism $r_{V} \colon
  \Terminal \to \mathcal{S}$ that corresponds to $V$ is perfect
  (cf.\ Section \ref{sec:BasicCover}).

  First, suppose that $V$ is located and that $a \ll b$.
  We must show that $r_{V}^{-}a \ll_{\Terminal} r_{V}^{-}b$.
  Let $U \subseteq \One$, and suppose that $r_{V}^{-}b \cov_{\Terminal} U$.
  Since $V$ is located, either $a \notin V$ or $b \in V$. In the
  former case, we have $r_{V}^{-}a \cov_{\Terminal} \emptyset$. In
  the latter case, we have $U = \One$, and so $r_{V}^{-}a
  \cov_{\Terminal} \One$. Thus $r_{V}^{-}a \ll_{\Terminal} r_{V}^{-}b$.

  Conversely, suppose that $r_{V}$ is perfect, and that $a \ll b$. 
  Then  $r_{V}^{-}a \ll_{\Terminal}
  r_{V}^{-}b$, so there exists
  $U \in \Fin{r_{V}^{-}b}$ such that $r_{V}^{-}a \cov_{\Terminal}U$. Either
  $U = \emptyset$ or $U = \One$. In the former case, we have
  $a \notin V$, and in the latter case, we have $b \in V$. 
\end{proof}

\begin{corollary}
  The located points of a locally compact formal topology
  $\mathcal{S}$ bijectively correspond to the perfect formal topology
  maps from $\Terminal$ to $\mathcal{S}$.
\end{corollary}

The following lemma will help the reader
understand the Definition \ref{def:Cut}.
\begin{lemma} \label{lem:SplittingInLK}
  Let $\mathcal{S}$ be a continuous basic cover. Then, a
  subset $V \subseteq S$ is splitting if and only if it satisfies the
  following properties:
  \begin{enumerate}
    \item $a \cov A\amp a \in V
      \implies A \meets V$,
    \item $a \in V \implies \left( \exists b \ll a \right) b \in V$
  \end{enumerate}
  for each $a \in S$ and $A \in \Fin{S}$.
\end{lemma}
\begin{proof} 
  Immediate from Remark \ref{rem:LKInd}.
\end{proof}

Motivated by Example \ref{eg:Dedekind} of extended Dedekind reals, we give a cut like
characterisation of a located subset.
\begin{definition}\label{def:Cut}
  Let $\mathcal{S}$ be a continuous basic cover.
  A pair $(L,U)$ of subsets of $S$ is called a \emph{cut}
  if
  \begin{multicols}{2}
  \begin{enumerate}
    \item\label{def:Cut1} $a \cov A \amp a \in U
      \implies A \meets U$,

    \item\label{def:Cut3} $a \in U \implies \left( \exists b \ll a \right) b \in U$,
    \item\label{def:Cut2} $a \cov A \subseteq L \implies
      a \in L$,

    \item\label{def:Cut4} $a \in L \implies \left( \exists A \gg a
      \right) A \subseteq L$,

    \item\label{def:Cut5} $a \ll b \implies a \in L \vee b \in U$,

    \item\label{def:Cut6} $L \cap U = \emptyset$,
   \end{enumerate}
  \end{multicols}
\noindent where $a,b \in S$ and $A \in \Fin{S}$. 
\end{definition}
\begin{remark}
  If a continuous basic cover $\mathcal{S}$ is equipped with
  a join semilattice structure $(S,0,\sqcup)$ that is compatible with the cover
  $\cov$ in the sense that
  \[
    \frac{ a \cov \emptyset}{a  \cov 0},
    \qquad 
    0 \cov \emptyset,
    \qquad
    \left\{ a,b \right\} \cov a \sqcup b,
    \qquad
    \frac{ \left\{ a,b \right\} \cov U}{a \sqcup b \cov U},
  \]
  then a cut $(L,U)$ can be defined in a symmetric way as follows:
  \begin{multicols}{2}
  \begin{enumerate}
    \item $0 \notin U$,
    \item $a \cov b \amp a \in U \implies b \in U$,
    \item $a \sqcup b \in U
      \implies a \in U \vee b \in U$,
    \item $a \in U \implies \left( \exists b \ll a \right) b \in U$,

    \item $0 \in L$,
    \item $a \cov b \amp b \in L \implies a \in L$,
    \item $a, b \in L
      \implies a \sqcup b \in L$,
    \item $a \in L \implies \left( \exists b \gg a \right)
      b \in L$,

    \item\label{it:located} $a \ll b \implies a \in L \vee b \in U$,

    \item\label{it:disjoint} $L \cap U = \emptyset$.
   \end{enumerate}
 \end{multicols}
\noindent In other words, a cut $(L,U)$ is a disjoint pair of a rounded prime upper set
 $U$ and a rounded ideal $L$ that is \emph{located} (namely, the condition \ref{it:located}).
\end{remark}

\begin{proposition} \label{prop:Cut}
  Let $\mathcal{S}$ be a continuous basic cover.
   Then, the located subsets of $\mathcal{S}$
  bijectively correspond to the cuts of $\mathcal{S}$ by
    $
    V \mapsto \left( L_{V}, V \right)
    $
  where
  \[
    L_{V} \defeql \left\{ a \in S \mid \left( \exists A \in
      \Fin{S}\right) a \ll A \amp A \cap V = \emptyset \right\}.
    \]
\end{proposition}
\begin{proof}
  In the following, the numbers 1, 2, \dots refer to the
  items in Definition \ref{def:Cut}.
  First, it is easy to show that 
  $\left( L_{V}, V \right)$ is a cut whenever $V$ is located, and
  if $\left(L, U \right)$  is a cut, then $U$ is a located subset.
We show that the correspondence is bijective. 
To this end, it suffices to show that for any cut $(L,U)$, we have
$L = L_{U}$. Let $a \in L$. By \ref{def:Cut4}, there exists
$A \in \Fin{L}$ such that 
$a \ll A$. Then $A \cap U = \emptyset$ by \ref{def:Cut6}.
  Hence $a \in L_{U}$. Conversely, let $a \in L_{U}$. Then, there
  exists $A \in \Fin{S}$ such that $a \ll A$ and $A \cap U =
  \emptyset$.
  Thus, there exists $C \in \Fin{S}$ such that $a \cov C$
  and $\left( \forall c \in C \right)\left( \exists a \in A \right) c
  \ll a$. By \ref{def:Cut5}, either $C \subseteq L$ or $A \meets U$.
  In the former case, we have $a \in L$ by \ref{def:Cut2}. The 
  latter case is a contradiction.
\end{proof}

\begin{corollary}\label{cor:CutPerfectPt}
  Let $\mathcal{S}$ be a locally compact formal topology.
  Then, the located points of $\mathcal{S}$
  bijectively correspond to the cuts $(L,U)$ of $\mathcal{S}$ that satisfies
  \begin{enumerate}
    \item $U \meets S$,
    \item $a,b \in U \implies \left( a \downarrow b \right) \meets U$.
  \end{enumerate}
\end{corollary}

It will be convenient to work with the following notion of cut.
\begin{lemma}\label{lem:CutAsScott}
  Let $\mathcal{S}$ be a continuous basic cover.  Then, the cuts of
  $\mathcal{S}$ bijectively correspond to the pairs $(\mathbb{L},U)$
  of subsets of $\Fin{S}$  and $S$ that satisfies
  {
  \setlength{\columnsep}{-65pt}
  \begin{multicols}{2}
  \begin{enumerate}
    \item $a \cov A \amp a \in U \Rightarrow A \meets U$,

    \item $a \in U \Rightarrow \left( \exists b \ll a \right) b \in
      U$,

    \item $\emptyset \in \mathbb{L}$,

    \item $A \cov B \amp B \in \mathbb{L} \Rightarrow A \in
      \mathbb{L}$,

    \item $A, B \in \mathbb{L} \Rightarrow \left(\exists C \in
      \mathbb{L} \right) A \ll C \amp B \ll C$,

    \item $a \ll b \Rightarrow \left\{ a \right\} \in \mathbb{L} \vee
      b \in U$,

    \item $\left\{ a \right\} \in \mathbb{L} \amp a \in  U \Rightarrow
      \bot$,
  \end{enumerate}
  \end{multicols}
\noindent where $a,b \in S$ and $A,B \in \Fin{S}$. 
}
\end{lemma}
\begin{proof}
  Given a cut $(L,U)$, the pair $(\Fin{L},U)$ satisfies
  the above conditions.  Conversely, if a pair $(\mathbb{L},U)$ satisfies
  those conditions, then $(\bigcup \mathbb{L},U)$ is a cut.
\end{proof}

\section{Patch topologies}\label{sec:Patch}
We give a geometric characterisation of the
patch topology of a stably locally compact formal topology.
Here, the patch topology of a stably locally compact formal topology
$\mathcal{S}$ is constructed as the space of its located points, i.e.\
the formal topology presented by a geometric theory whose models are
the located points of $\mathcal{S}$.
We also discuss patch topologies of stably compact formal
topologies and spectral formal topologies, which constitute
important subclasses of stably locally compact formal topologies.

\subsection{Stably locally compact formal topologies}
\begin{definition}\label{def:StablyLK}
  A locally compact formal topology $\mathcal{S}$ is \emph{stably locally compact}
  if $a \ll a' \amp b \ll b' \implies a \downarrow b \ll a' \downarrow b'$
  for all $a,a',b,b' \in S$.
\end{definition}
\begin{remark}
  The notion of stably locally compact formal topology corresponds
  to that of stably locally compact locale discussed by
  \citet{escardo2001regular}.
  \citet{johnstone-82} requires stably locally compact locales to be
  compact as well. A formal topology that corresponds to this notion
  is called stably compact in this paper; see Definition \ref{defSK}.
\end{remark}

Important examples of stably locally compact formal topologies are
locally compact regular formal topologies.
Recall that a formal topology $\mathcal{S}$ is \emph{regular} if 
\[
  a \cov \left\{ b \in S \mid b \lll a \right\}
\]
for all $a \in S$, where $b \lll a \defeqiv S \cov b^{*} \cup \left\{ a
\right\}$ and 
$
b^{*} \defeql \left\{ c \in S \mid c \downarrow b \cov \emptyset
\right\}$.
We extend $\lll$ to the subsets of $S$ by 
$
U \lll V \defeqiv S \cov U^{*} \cup V,
$
where $U^{*} \defeql \bigcap_{a \in U} a^{*}$.
 
\begin{lemma}[cf.\ {\citet[Lemma 4.2]{escardo2001regular}}]\label{lem:BoundedWcImpliesWb}
  Let $\mathcal{S}$ be a formal topology. For any $U,V \subseteq S$
  we have
  $
  U \ll S \,\amp\, U \lll V \implies U \ll V.
  $
\end{lemma}
\begin{proof}
  Suppose that $U \ll S$ and  $U \lll V$.
  Let $W \subseteq S$ such that $V \cov W$. Then $S \cov U^{*} \cup V
  \cov  U^{*} \cup W$. Thus there exists $W_0 \in
  \Fin{W}$ such that $U \cov U^{*} \cup W_0$. Hence $U \cov \left(
  U^{*} \cup W_0 \right) \downarrow U \cov \left( U^{*} \downarrow U
  \right) \cup \left( W_0 \downarrow U \right) \cov W_0$. Therefore 
  $U \ll V$.
\end{proof} 
\begin{proposition}[{\citet[Lemma 4.3]{escardo2001regular}}]
   Every locally compact regular formal topology is stably locally compact.
\end{proposition}
\begin{proof}
  Let $\mathcal{S}$ be a locally compact regular formal topology.
  Suppose that $a \ll a'$ and $b \ll b'$. Since $\mathcal{S}$ is
  regular,
  we have $a \lll a'$ and $b \lll b'$. Then, $a \downarrow b \lll a'
  \downarrow b'$. Since $a \downarrow b \ll S$, we have $a
  \downarrow b \ll a' \downarrow b'$ by Lemma
  \ref{lem:BoundedWcImpliesWb}.
\end{proof}

In what follows, we fixed a stably locally compact
formal topology $\mathcal{S}$ .
\begin{definition} \label{def:Patch}
The patch of $\mathcal{S}$ is a formal topology $\Patch(\mathcal{S})$
presented by a geometric theory $T_{P}$ over the propositional
symbols
\[
  P_{P} \defeql \left\{ \lcut(A) \mid A \in \Fin{S} \right\}
  \cup \left\{ \rcut(a) \mid a \in S \right\}
\] 
with the following axioms:
\begin{enumerate}
    \myitem[(R1)]\label{r1}  $ \top \vdash \bigvee \left\{ 
   \rcut(a) \mid a \in S \right\}$ 
  
  \myitem[(R2)]\label{r2}  $\rcut(a) \wedge \rcut(b) \vdash \bigvee
  \left\{ \rcut(c) \mid c \in a \downarrow b  \right\}$ 
  
  \myitem[(R3)]\label{r3}  $\rcut(a)  \vdash \bigvee \left\{ 
  \rcut(b) \mid b \in A \right\} \qquad (a \cov A \in \Fin{S})$
  
  \myitem[(R4)]\label{r4} $\rcut(a)  \vdash \bigvee_{b \ll a} \rcut(b)$ 
  
  \myitem[(L1)]\label{l1} $ \top  \vdash \lcut(\emptyset) $
  
  \myitem[(L2)]\label{l2} $\lcut(B) \vdash \lcut(A) \qquad (A \cov B)$
  
  \myitem[(L3)]\label{l3} $\lcut(A) \wedge \lcut(B)
        \vdash \bigvee \left\{ \lcut(C) \mid A \ll C \amp B \ll C \right\} $
  
  \myitem[(Loc)]\label{Loc} $\top \vdash  \lcut(\left\{ a \right\}) \vee \rcut(b)
  \qquad (a \ll b)$
  
  \myitem[(D)]\label{D}  $\lcut(\left\{ a \right\}) \wedge \rcut(a) \vdash \bot$
\end{enumerate}
We write $\cov_{P}$ for the cover of $\Patch(\mathcal{S})$.
\end{definition}
By Corollary \ref{cor:CutPerfectPt} and Lemma \ref{lem:CutAsScott},
the models of the above theory correspond to the
located points of $\mathcal{S}$.

\begin{remark}\label{rem:TheoryT}
  The geometric theory over $\left\{ \rcut(a) \mid a \in S \right\}$
  with the axioms \ref{r1} -- \ref{r4} presents the topology
  $\mathcal{S}$ \cite[Proposition 4.1.12]{Fox05}. Moreover, it is not
  hard to show that the geometric theory over $\left\{ \lcut(A) \mid A
    \in \Fin{S} \right\}$ with the axioms \ref{l1} -- \ref{l3} 
  presents the Scott topology $\Scott(\mathcal{S})$ of
  $\Sat{\mathcal{S}}$.  
\end{remark} 
The following lemma generalises \ref{Loc}.
\begin{lemma}
  \label{lem:LittleFact}
  For any $A,B \in \Fin{S}$
  \[
    A \ll B \implies
  \Fin{P_{P}} \cov_{P} \left\{ \left\{ \lcut(A) \right\} \right\}
  \cup \left\{ \left\{ \rcut(b) \right\} \mid b \in B\right\}.
  \]
\end{lemma}
\begin{proof}
  Suppose that $A \ll B$. Then, there exists $C \in \Fin{S}$ such
  that $A \cov C$ and $\left( \forall c \in C \right) \left( \exists b
  \in B \right) c \ll b$. By \ref{Loc} and \ref{l3}, we have
  $\Fin{P_{P}} \cov_{P} \left\{ \left\{ \lcut(C) \right\} \right\}
  \cup \left\{ \left\{ \rcut(b) \right\} \mid b \in B\right\}$.
  Hence, $\Fin{P_{P}} \cov_{P} \left\{ \left\{ \lcut(A) \right\} \right\}
  \cup \left\{ \left\{ \rcut(b) \right\} \mid b \in B\right\}$ by
  \ref{l2}.
\end{proof}

We show that $\Patch(\mathcal{S})$ is a locally compact regular formal
topology. To this end, define a function
  $
  \wb_{P} \colon \Fin{P_P} \to \Pow{\Fin{P_P}}
  $
by induction on $\Fin{P_P}$:
\begin{align*}
  \wb_{P}(\emptyset)
  &\defeql
  \left\{ \left\{ \rcut(a) \right\}  \mid a \ll S \right\},\\
  \wb_{P}(\elPT{A} \cup \left\{ \rcut(a) \right\})
  &\defeql
  \left\{ \elPT{B} \cup \left\{ \rcut(b) \right\} \mid
  \elPT{B} \in \wb_{T}(\elPT{A}) \amp  b \ll a \right\},\\
  \wb_{P}(\elPT{A} \cup \left\{ \lcut(A) \right\})
  &\defeql
  \left\{ \elPT{B} \cup \left\{ \lcut(B) \right\} \mid
  \elPT{B} \in \wb_{T}(\elPT{A}) \amp  A \ll B  \right\}.
\end{align*}

\begin{lemma}\label{lem:wbT} 
  For each $\elPT{A},\elPT{B} \in \Fin{P_P}$, we have
  \begin{enumerate}
    \item\label{lem:wbT0} $\elPT{B} \in \wb_{P}(\elPT{A}) \implies
      \elPT{B} \cov_{P} \elPT{A}$,
    \item\label{lem:wbT1} $\elPT{B} \in \wb_{P}(\elPT{A}) \implies
      \elPT{B} \ll \elPT{A}$,
    \item\label{lem:wbT2} $\elPT{B} \in \wb_{P}(\elPT{A}) \implies
      \elPT{B} \lll \elPT{A}$,
    \item\label{lem:wbT3} $\elPT{A} \cov_{P} \wb_{P}(\elPT{A})$.
  \end{enumerate}
\end{lemma}
\begin{proof}
  \ref{lem:wbT0}.
  By a straightforward induction on $\elPT{A}$.

  \medskip

  \noindent\ref{lem:wbT1}.  We show that for any $\elPT{A} \in \Fin{P_P}$ and
  $\mathcal{U} \subseteq \Fin{P_P}$, 
  \[
    \elPT{A} \cov_{P} \mathcal{U} \implies \left( \forall \elPT{B} \in
    \wb_{P}(\elPT{A}) \right) \left( \exists\, \mathcal{U}_{0} \in
    \Fin{\mathcal{U}} \right) \elPT{B} \cov_{P} \mathcal{U}_{0}
  \]
  by induction on $\cov_{P}$.
  Put
  \[
    \Phi(\elPT{A}) \equiv \left( \forall \elPT{B} \in
    \wb_{P}(\elPT{A}) \right) \left( \exists\, \mathcal{U}_{0} \in
    \Fin{\mathcal{U}} \right) \elPT{B} \cov_{P} \mathcal{U}_{0}.
  \]
  If $\elPT{A} \cov_{P} \mathcal{U}$ is derived by (reflex)-rule,
  then we have $\Phi(\elPT{A})$ from \ref{lem:wbT0}.
  Next, suppose that $\elPT{A} \cov_{P} \mathcal{U}$ is derived by ($\leq$)-rule.
  Let $\elPT{A}' \in \Fin{P_{P}}$ such that $\elPT{A} \supseteq \elPT{A}'$,
  and assume that $\Phi(\elPT{A}')$.
  Let $\elPT{B} \in \wb_{P}(\elPT{A})$. By the definition of
  $\wb_{P}$, there exists $\elPT{C} \in \wb_{P}(\elPT{A}')$ such that
  $\elPT{C} \subseteq \elPT{B}$. Thus, there exists $\mathcal{U}_{0}
  \in \Fin{\mathcal{U}}$ such that $\elPT{C} \cov_{P}
  \mathcal{U}_{0}$.  Hence $\elPT{B} \cov_{P} \mathcal{U}_{0}$.

  Finally, for (infinity)-rule, we check each axiom of $T_P$.

  \medskip

  \noindent\ref{r1} Assume
  $\Phi(\elPT{A} \cup \left\{ \rcut(a) \right\})$ for all $a \in S$.
  Let $\elPT{B} \in \wb_{P}(\elPT{A})$. By the definition of
  $\wb_{P}(\elPT{A})$, there
  exists $a \ll S$ such that $\rcut(a) \in \elPT{B}$. Thus, there exist
  $\left\{ a_{0},\dots,a_{n-1} \right\}$ and
  $\left\{ b_{0},\dots,b_{n-1} \right\}$ such that 
  \[
    a \cov \left\{ a_{0},\dots,a_{n-1} \right\} \amp \left( \forall i
    < n \right) a_{i} \ll b_{i}.
  \]
  For each $i < n$, since $\elPT{B} \cup \left\{ \rcut(a_{i}) \right\} \in
  \wb_{P}(\elPT{A} \cup \left\{ \rcut(b_{i}) \right\})$, 
  there exists $\mathcal{U}_{i} \in \Fin{\mathcal{U}}$ such that
  $\elPT{B} \cup \left\{ \rcut(a_{i}) \right\} \cov_{P}
  \mathcal{U}_{i}$. Then, by \ref{r3}, we have
  \[
    \elPT{B} \cov_{P} \left\{ \elPT{B} \cup \left\{ \rcut(a_{i}) \right\} \mid i <
  n \right\} \cov_{P} \bigcup_{i < n} \mathcal{U}_{i}.
  \]
  \noindent\ref{r2} Assume $\Phi(\elPT{A} \cup \left\{ \rcut(c) \right\})$
  for all $c \in a \downarrow b$. 
  Let $\elPT{B} \in
  \wb_{T}(\elPT{A} \cup \left\{ \rcut(a), \rcut(b) \right\})$.
  Then, there exists $\elPT{B}' \in \wb_{P}(\elPT{A})$, $a' \ll a$ and  $b' \ll b$
  such that $\elPT{B} = \elPT{B}' \cup \left\{ \rcut(a'), \rcut(b') \right\}$.
  Since $\mathcal{S}$ is stably locally compact, we have
  $a' \downarrow b' \ll a \downarrow b$. Thus, there exist
  $\left\{ c_{0},\dots,c_{n-1} \right\}$ and
  $\left\{ c_{0}',\dots,c_{n-1}' \right\} $
  such that  $a' \downarrow b' \cov \left\{ c_{0},\dots,c_{n-1}
  \right\}$ and $\left( \forall i < n \right) c_{i} \ll c_{i}' \in a \downarrow b$.
  Then, for each $i < n$, there exists
  $\mathcal{U}_{i} \in \Fin{\mathcal{U}}$ such that $\elPT{B}' \cup \left\{
  \rcut(c_{i}) \right\} \cov_{P}\mathcal{U}_{i}$. By
  \ref{r2} and \ref{r3}, we have
  \[
    \elPT{B} \cov_{P} \left\{ \elPT{B}' \cup \left\{ \rcut(c') \right\}
               \mid c' \in a' \downarrow b' \right\} 
    \cov_{P}
    \left\{ \elPT{B}' \cup \left\{ \rcut(c_{i}) \right\} \mid i < n \right\} 
    \cov_{P}
    \bigcup_{i < n} \mathcal{U}_{i}.
  \]
  \noindent\ref{r3} Suppose that $a \cov \left\{ a_{0},\dots,a_{n-1} \right\}$,
  and assume $\Phi(\elPT{A} \cup \left\{ \rcut(a_{i}) \right\})$
  for each $i < n$. Let $\elPT{B} \in \wb_{P}(\elPT{A} \cup \left\{
  \rcut(a) \right\})$. Then, there exists $\elPT{B}' \in
  \wb_{P}(\elPT{A})$ and $b \ll a$ such that $\elPT{B} = \elPT{B}'
  \cup \left\{ \rcut(b) \right\}$.
  Hence, there exists $\left\{ b_{0},\dots,b_{m-1} \right\}$ such
  that $b \cov \left\{ b_{j} \mid j < m \right\}$ and $\left
  ( \forall j < m \right)\left( \exists i < n \right)
  b_{j} \ll a_{i}$. Then, for each
  $j < m$ there exists $\mathcal{U}_{j} \in \Fin{\mathcal{U}}$ such
  that $\elPT{B}' \cup \left\{ \rcut(b_{j}) \right\} \cov_{P}
  \mathcal{U}_{j}$. Thus $\elPT{B} \cov_{P}
  \bigcup_{j < m} \mathcal{U}_{j}$ by \ref{r3}.

  \medskip

  \noindent\ref{r4} Similar to the case \ref{r3}.

  \medskip

  \noindent Verification of the axioms \ref{l1}, \ref{l2}, and  \ref{l3} is straightforward.

  \medskip

  \noindent\ref{Loc} Suppose that $a \ll b$. Assume $\Phi(\elPT{A} \cup \left\{
   \lcut(\left\{ a \right\}) \right\})$ and $\Phi(\elPT{A} \cup
   \left\{ \rcut(b) \right\})$.
   Let $\elPT{B} \in \wb_{P}(\elPT{A})$. Since $a \ll b$, there
   exist $A,C \in \Fin{S}$ such that $a \ll A \ll C$ and $\left(
   \forall c \in C \right) c \ll b$. By Lemma \ref{lem:LittleFact},
   we have 
   $\Fin{P_{P}} \cov_{P} \left\{ \left\{ \lcut(A) \right\} \right\}
   \cup \left\{ \left\{ \rcut(c) \right\} \mid c \in C\right\}$.
   Thus,
   \[
     \elPT{B} \cov_{P}
     \left\{\elPT{B} \cup  \left\{ \lcut(A) \right\} \right\}
     \cup \left\{ \elPT{B} \cup \left\{ \rcut(c) \right\} \mid c \in
   C\right\}.
   \]
   Note that the right hand side is finitely enumerable, and that
   $ \elPT{B} \cup  \left\{ \lcut(A) \right\} \in
   \wb_{P}(\elPT{A} \cup \left\{ \lcut(\left\{ a \right\}) \right\})$ and 
   $\elPT{B} \cup \left\{ \rcut(c) \right\} 
   \in \wb_{P}(\elPT{A} \cup \left\{ \rcut(b) \right\})$
   for each  $c \in C$. Hence, by the assumption, there
   exists $\mathcal{U}_{0} \in \Fin{\mathcal{U}}$ such that
   $\elPT{B} \cov_{P} \mathcal{U}_{0}$.

   \medskip

   \noindent\ref{D} Let $a \in S$, and let
   $\elPT{B} \in \wb_{P}(\elPT{A} \cup
   \left\{ \lcut(\left\{ a \right\}), \rcut(a) \right\})$.
   Then, there exists
   $\elPT{B}' \in \wb_{P}(\elPT{A})$, $A \gg a$,
   and  $b \ll a$ such that $\elPT{B} = \elPT{B}' \cup \left\{
   \lcut(A), \rcut(b) \right\}$. By \ref{l2} and \ref{D}, we have 
   $\elPT{B} \cov_{P} \left\{ \lcut(\left\{ b \right\}), \rcut(b)  \right\} \cov_{P} \emptyset$.

  \medskip
  \noindent\ref{lem:wbT2}. By induction on $\elPT{A}$.

  \smallskip
  \noindent$\elPT{A} = \emptyset$:
  Let $\elPT{B} \in \wb_{P}(\emptyset)$. Since $\emptyset$ is
  the top element, we have $\Fin{P_P} \cov_{P}  \elPT{B}^{*} \cup \left\{ \emptyset \right\}$.

  \smallskip

  \noindent$\elPT{A} = \elPT{A}' \cup \left\{ \rcut(a) \right\}$:
  Let $\elPT{B} \in \wb_{P}(\elPT{A})$. Then, there exists $\elPT{B}'
  \in \wb_{P}(\elPT{A}')$ and
  $b \ll a$ such that $\elPT{B} = \elPT{B}' \cup \left\{ \rcut(b) \right\}$.
  By induction hypothesis and \ref{Loc}, we have 
  \[
    \Fin{P_P}
    \cov_{P} \elPT{B}'^{*} \cup \left\{ \elPT{A}' \cup
    \left\{ \lcut(\left\{ b \right\}) \right\}, \elPT{A}' \cup
    \left\{ \rcut(a) \right\}\right\} \\
    \cov_{P} \elPT{B}^{*} \cup \left\{ \elPT{A}' \cup
    \left\{ \lcut(\left\{ b \right\}) \right\}, \elPT{A} \right\}.
    \]
   Since  $\elPT{A}' \cup \left\{ \lcut(\left\{ b \right\}) \right\}
   \in \elPT{B}^{*}$ by \ref{D}, we have  $\Fin{P_P} \cov_{P} \elPT{B}^{*}
   \cup \left\{ \elPT{A} \right\}$.

  \smallskip

  \noindent$\elPT{A} = \elPT{A}' \cup \left\{ \lcut(A) \right\}$:
   Similar to the previous case, using Lemma \ref{lem:LittleFact}.

  \medskip
  \noindent\ref{lem:wbT3}.
  By induction on $\elPT{A}$, using \ref{r1}, \ref{r4}, and
  \ref{l3}.
\end{proof}

\begin{proposition}\label{prop:PatchLKReg}
  $\Patch(\mathcal{S})$ is locally compact regular.
\end{proposition}

We show that $\Patch(\mathcal{S})$ is the best locally compact regular
approximation of $\mathcal{S}$ (cf.\ Theorem \ref{thm:CorefPatchSLK}). To this end,
define a formal topology map $\varepsilon_{P} \colon \Patch(\mathcal{S}) \to
\mathcal{S}$ by
\[
  \elPT{A} \mathrel{\varepsilon_{P}} a \defeqiv \elPT{A} \cov_{P} \left\{
    \rcut(a) \right\}.
\]
Note that $\varepsilon_{P}$ is indeed a formal topology map by Remark \ref{rem:TheoryT}.

We recall the following lemma, which simplifies some of our developments.
\begin{lemma}[{\citet[Theorem 4.3]{Palmgren:Predicativiy_problems_in_point-free_topology}}]\label{lem:RegMax}
  Let $r,s \colon \mathcal{S} \to \mathcal{S}'$ be formal topology maps
  where $\mathcal{S}'$ is regular.
  Then
    $
    r \leq s \implies s \leq r.
    $
\end{lemma}

\begin{lemma}\label{lem:Core}
  \leavevmode
  \begin{enumerate}
   \item \label{lem:CorefPerfect}
  $\varepsilon_{P} \colon \Patch(\mathcal{S}) \to \mathcal{S}$ is  a
  perfect map.

  \item\label{lem:CorefMono}
  $\varepsilon_{P} \colon \Patch(\mathcal{S}) \to \mathcal{S}$ is
  a monomorphism in $\FTop$.
  \end{enumerate}
\end{lemma}
\begin{proof}
  \ref{lem:CorefPerfect}.\
  Immediate from the definition of $\wb_{P}$ and interpolation of $\ll$.

  \medskip
 \noindent\ref{lem:CorefMono}.
  Let $r,s \colon \mathcal{S}' \to \Patch(\mathcal{S})$ be formal
  topology maps such that $\varepsilon_{P} \circ r = \varepsilon_{P}
  \circ s$. We must show that $r = s$. By the definition of
  $\varepsilon_{P}$, and since $\Patch(\mathcal{S})$ is regular, it
  suffices to show that
  $
  r^{-}\left\{ \lcut(A) \right\} \cov' s^{-}\left\{ \lcut(A) \right\}
  $
  for each $A \in \Fin{S}$. Moreover, by \ref{l3}, it suffices to show that
    $
    A \ll B \implies r^{-}\left\{ \lcut(B) \right\}
    \cov' s^{-}\left\{ \lcut(A) \right\}.
    $
    Suppose that $A \ll B$. By Lemma \ref{lem:LittleFact}, we have
    $\Fin{P_{P}} \cov_{P} \left\{ \left\{ \lcut(A) \right\} \right\}
  \cup \left\{ \left\{ \rcut(b) \right\} \mid b \in B\right\}$. Thus
    $S'
    \cov'
    s^{-} \left\{ \lcut(A) \right\} \cup \bigcup_{b \in B} s^{-}
    \left\{ \rcut(b) \right\}
    \cov'
    s^{-} \left\{ \lcut(A) \right\} \cup \bigcup_{b \in B} r^{-}
    \left\{ \rcut(b) \right\}$.
    Hence, 
    \begin{align*}
      r^{-}\left\{ \lcut(B) \right\}
      &\cov' r^{-}\left\{ \lcut(B) \right\} \downarrow
      \left( s^{-} \left\{ \lcut(A) \right\}
      \cup \bigcup_{b \in B} r^{-} \left\{ \rcut(b) \right\} \right)\\
      &\cov' s^{-} \left\{ \lcut(A) \right\} \cup
      \bigcup_{b \in B}  r^{-}\left\{ \lcut(B),\rcut(b) \right\}\\
      &\cov' s^{-} \left\{ \lcut(A) \right\}. && \text{(by \ref{D})} \qedhere
    \end{align*}
\end{proof}

The following lemma is useful in Proposition \ref{prop:UnivPropPatch} below.
\begin{lemma}[{\citet[Lemma 4.4]{escardo2001regular}}]\label{lem:cobounded}
  A formal topology map $r \colon \mathcal{S} \to \mathcal{S}'$
  between locally compact regular formal topologies is perfect if and
  only if it is cobounded, i.e.\ $U \ll' S' \implies r^{-}U \ll S$ for all
  $U \subseteq S'$.
\end{lemma}
\begin{proof}
  Suppose that $r$ is perfect, and assume that $U
  \ll' S'$.  Then $r^{-}U \ll r^{-}S' \cov S$. Thus $r^{-}U \ll S$.
  Conversely, suppose that $r$ is cobounded, and assume that $a \ll' b$. 
  Since $\mathcal{S}'$ is regular, we have $a \lll' b$.
  Since $r$ is cobounded and the relation $\lll$ is preserved by every
  formal topology map, we have $r^{-} a\ll S$ and $r^{-}a \lll
  r^{-}b$. Hence $r^{-} a \ll r^{-}b$ by Lemma \ref{lem:BoundedWcImpliesWb}.
  Therefore, $r$ is perfect.
\end{proof}

\begin{proposition}\label{prop:UnivPropPatch}
  For any locally compact regular formal topology $\mathcal{S}'$ and a
  perfect map $r \colon \mathcal{S}' \to \mathcal{S}$, there exists a
  unique perfect map $\tilde{r} \colon \mathcal{S}' \to
  \Patch(\mathcal{S})$ such that $\varepsilon_{P} \circ
  \tilde{r} = r$.
\end{proposition}
\begin{proof}
  Let $r \colon \mathcal{S}' \to \mathcal{S}$ be a perfect map from a
  locally compact regular formal topology $\mathcal{S}'$.
  We define a formal topology map $\tilde{r} \colon \mathcal{S}' \to
  \Patch(\mathcal{S})$ by specifying its action on generators as follows:
  \begin{align}
    \begin{aligned}\label{def:uniqExt}
    b \mathrel{\tilde{r}} \left\{ \rcut(a) \right\}
    &\defeqiv b \mathrel{r} a, \\
    b \mathrel{\tilde{r}} \left\{ \lcut(A) \right\}
    &\defeqiv 
    \left( \exists B \gg A \right) b \in \left( r^{-} B \right)^{*}.
    \end{aligned}
  \end{align}
  We must show that $\tilde{r}$ respects the axioms of $T_{P}$ in
  the sense of \eqref{eq:PreservAx}.
    Since $r \colon \mathcal{S}' \to \mathcal{S}$ is a formal topology
    map, $\tilde{r}$ respects the axioms \ref{r1} -- \ref{r4}.
    It is also straightforward to show that $\tilde{r}$
    respects the axioms \ref{l1} -- \ref{l3} and \ref{D}. It remains
    to check \ref{Loc}.
    Suppose that $a \ll b$. There exists $A \in \Fin{S}$ such
      that $a \ll A \ll b$. Since $r$ is a  perfect map, we have
      $r^{-}A \ll r^{-}\left\{ b \right\}$, and since $\mathcal{S}'$
      is regular, we obtain $r^{-}A \lll r^{-}\left\{ b \right\}$. Thus,
      \[
        S' \cov' \left( r^{-}A \right)^{*} \cup r^{-}\left\{ b
        \right\} \cov' \tilde{r}^{-}\left\{ \lcut(\left\{ a \right\})
        \right\} \cup \tilde{r}^{-}\left\{ \rcut(b) \right\}.
      \]
      Hence, $\tilde{r}$ gives rise to a formal topology map
      from $\mathcal{S}'$  to $\mathcal{S}$, 
      which obviously satisfies $\varepsilon_{P} \circ \tilde{r} = r$.
      Moreover, $\tilde{r}$ is a unique such morphism by Lemma
      \ref{lem:Core}.\ref{lem:CorefMono}.

      Lastly, to see that $\tilde{r}$ is a perfect map,
      it suffices to show that $\tilde{r}$ is cobounded by Lemma
      \ref{lem:cobounded}. 
      Since $\emptyset$ is the top element
      of $\Patch(\mathcal{S})$, it is enough to show that
      $\tilde{r}^{-} \elPT{A} \ll' S'$ for each $\elPT{A} \in
      \wb_{P}(\emptyset)$. But this is equivalent to $a \ll S \implies
      r^{-}\left\{ a \right\} \ll' S'$, which follows from the fact
      that $r$ is perfect.
\end{proof}

\begin{proposition}
  If $\mathcal{S}$ is a locally compact regular formal topology, then
  $\varepsilon_{P} \colon \Patch(\mathcal{S}) \to \mathcal{S}$ is an
  isomorphism.
\end{proposition}
\begin{proof}
  The identity map $\id_{\mathcal{S}} \colon \mathcal{S} \to \mathcal{S}$ uniquely
  extends to a perfect map $s \colon \mathcal{S} \to \Patch(\mathcal{S})$ such
  that $\varepsilon_{P} \circ s = \id_{\mathcal{S}}$. Thus,
  $\varepsilon_{P}$ is a split epi. Since $\varepsilon_{P}$ is a
  monomorphism  by Lemma \ref{lem:Core}.\ref{lem:CorefMono}, it is an isomorphism.
\end{proof}

\begin{theorem}[{\citet[Theorem 5.8]{escardo2001regular}}]\label{thm:CorefPatchSLK}
  The patch construction exhibits the category of
  locally compact regular formal topologies
  and perfect maps as a coreflective subcategory of the category
  of stably locally compact formal topologies and perfect maps.
\end{theorem}

\subsection{Stably compact formal topologies}
\begin{definition}\label{defSK}
  A formal topology $\mathcal{S}$ is \emph{compact} if $S \ll S$, and
  it is  \emph{stably compact} if it is stably locally compact and
  compact. 
\end{definition}

In a compact regular formal topology, the relations $\ll$ and $\lll$
coincide \cite[Chapter IIV, Section 3.5]{johnstone-82}. Thus, every compact regular formal topology is locally
compact, and hence stably compact. Since the relation $\lll$ is
preserved by any formal topology map, every morphism between
compact regular formal topologies is perfect. Thus, the category of
compact regular formal topologies is a full subcategory of the
category of stably compact formal topologies and perfect maps.

\begin{lemma}[{\citet[Lemma 2.6]{escardo2001regular}}]
  If $r \colon \mathcal{S} \to \mathcal{S}'$ is a perfect map and 
  $\mathcal{S}'$ is compact, then $\mathcal{S}$ is compact.
\end{lemma}
\begin{proof}
  Immediate.
\end{proof}

\begin{theorem}[{\citet[Corollary
  5.9]{escardo2001regular}}]\label{thm:CorefPatchSLKC}
  The patch construction exhibits the category of
  compact regular formal topologies
  as a coreflective subcategory of the category
  of  stably  compact formal topologies and perfect maps.
\end{theorem}
Note that Theorem \ref{thm:CorefPatchSLKC} is one the main results by
\citet{CoquandZhangPrediativePatch}.

\subsubsection{De Groot duals}
One of the most interesting aspects of stably compact spaces is
\emph{de Groot duality}. Here, \emph{stably compact spaces} are the
spaces which arise as the formal points of stably compact formal
topologies.
It is well known that stably compact spaces are closed under taking de Groot
duals, and various constructions on stably compact spaces exhibit
interesting properties involving de Groot duals; see
\citet{GoubaultLarrecq-ModelofChoice}.
We give a geometric account of one of such properties proved by
\citet{escardoLawsonDual} in locale theory that the patch of stably
compact locale is isomorphic to the patch of its de Groot dual.

\begin{definition}\label{def:deGroot}
  Let $\mathcal{S}$ be a stably compact formal topology.
  The \emph{de Groot dual} of $\mathcal{S}$ is the formal topology
  $\mathcal{S}^{d}$ presented by a geometric theory $T^{d}$ over
  the propositional symbols
  \[
    P_{d} \defeql \left\{ \lcut(A) \mid A \in \Fin{S} \right\}
  \] 
  with the axioms \ref{l1} -- \ref{l3} of the theory $T_{P}$ (cf.\
  Definition \ref{def:Patch}), and the following additional axioms:
  \[
    \lcut(A) \vdash \bigvee \left\{ \lcut(A_{i}) \mid i < n \right\}
  \] 
  for each
  $\wayabove A \subseteq \bigvee^{F} \left\{
    \wayabove A_{0},
  \dots, \wayabove A_{n-1} \right\}$,
  where
  $\bigvee^{F} \left\{ \wayabove A_{0},
  \dots, \wayabove A_{n-1} \right\}$ is the finite join of
  Scott open filters $\wayabove A_{0}, \dots,
  \wayabove A_{n-1}$, i.e.\
  \[
    {\bigvee_{\mathclap{i < n}}}^{F}\!\!\wayabove A_{i}
    \defeql \left\{ B \in \Fin{S} 
    \mid  \left( \exists B_{0} \gg A_{0} \right) \cdots
    \left( \exists B_{n-1}\! \gg \!
    A_{n-1}\right) B_{0} \downarrow \cdots \downarrow B_{n-1} \ll B
  \right\}.
  \]
  The empty join is defined to be the smallest Scott open filter
  $\left\{ A \in \Fin{S} \mid S \cov A \right\}$.
\end{definition}

\begin{remark}
  Note that $\wayabove A$ is (a base of) a Scott open filter on
  $\Sat{\mathcal{S}}$ since $\mathcal{S}$ is stably compact.
  Since the axioms \ref{l1} -- \ref{l3} presents the Scott
  topology $\Sigma(\mathcal{S})$ of $\Sat{\mathcal{S}}$,
  the topology $\mathcal{S}^{d}$ is a subtopology of $\Sigma(\mathcal{S})$.
\end{remark}

In what follows, we fix a stably compact formal topology
$\mathcal{S}$. We mainly work with the following equivalent
set of axioms:
\begin{enumerate}
  \myitem[(d1)]\label{D1} $ \top  \vdash \lcut(\emptyset) $
  
  \myitem[(d2)]\label{D2} $\lcut(A) \wedge \lcut(B) \vdash \lcut(A
  \cup B)$
  
  \myitem[(d3)]\label{D3} $\lcut(A)
        \vdash \bigvee \left\{ \lcut(B) \mid A \ll B \right\} $

  \myitem[(d4)]\label{D4}
       $\lcut(A) \vdash
       \bigvee \left\{ \lcut(A_{i}) \mid i < n \right\}
       \qquad 
       \left(\wayabove A \subseteq \bigvee^{F} \left\{
         \wayabove A_{0},
       \dots, \wayabove A_{n-1} \right\}\right)$
\end{enumerate}

The following two lemmas serve to show that the de Groot dual of a stably
compact formal topology is also stably compact.
\begin{lemma}
  The topology $\mathcal{S}^{d}$ is isomorphic to the
  formal topology $\mathcal{S}^{D} = (\Fin{S},\cov^{D},\leq^{D})$,
  where $A \leq^{D} B \defeqiv B \subseteq A$ and the cover $ \cov^{D}$
  is inductively generated by the following axiom-set:
  \begin{enumerate}
    \myitem[\textup{(D3)}]\label{D3'}
        $A \cov^{D} \wayabove A$,

    \myitem[\textup{(D4)}]\label{D4'}
    $A \cov^{D}  \left\{ A_{0}, \dots, A_{n-1} \right\}
       \qquad 
       (\wayabove A \subseteq \bigvee^{F} \left\{
         \wayabove A_{0},
       \dots, \wayabove A_{n-1} \right\})$.
\end{enumerate}
\end{lemma}
\begin{proof}
  Immediate from the axioms of  $T^{d}$.
\end{proof}

\begin{lemma} \label{lem:DeGrootSK}
  The topology $\mathcal{S}^{D}$ satisfies
  \[
    A \cov^{D} \mathcal{U} \iff \left( \forall B \in \Fin{S}\right)
    \Bigr[ 
    B \gg A \implies \left( \exists \;\mathcal{U}_{0} 
    \in \Fin{\mathcal{U}} \right) \wayabove B \subseteq
  { \bigvee_{\mathclap{C \in  \mathcal{U}_{0}}}}^{F} \! \wayabove C \Bigr]
  \] 
  for all $A \in \Fin{S}$ and $\mathcal{U} \subseteq \Fin{S}$.
\end{lemma}
\begin{proof}
  By induction on $\cov^{D}$. 
\end{proof}
\begin{corollary}
  The formal topology $\mathcal{S}^{D}$ is stably compact, and so is $\mathcal{S}^{d}$.
\end{corollary}
\begin{proof} 
  By Lemma \ref{lem:DeGrootSK} and axioms \ref{D3'} and 
  \ref{D4'}, the topology $\mathcal{S}^{D}$ is locally compact.
  It is also compact because
  $\emptyset \ll \emptyset$.
  Lastly, $\mathcal{S}^{D}$ is stably compact
  since $A \ll B \amp A' \ll B' \implies A \cup A' \ll B \cup B'$.
\end{proof}

The de Groot dual of a locale $X$ is defined to be
the collection of Scott open filters on $X$ ordered by inclusion
 \cite{escardoLawsonDual}. The following lemma says that
Definition \ref{def:deGroot} is correct.
\begin{proposition}\label{prop:deGrootSOF}
  There exists a bijective correspondence between the saturated
  subsets of $\mathcal{S}^{D}$ and the Scott open filters on
  $\Sat{\mathcal{S}}$. Specifically, 
  \begin{enumerate}
    \item 
  if $\mathcal{U} \subseteq
  \Fin{S}$ is a saturated subset of $\mathcal{S}^{D}$, then
  $\bigcup_{A \in \mathcal{U}} \wayabove A$ is a Scott open filter,
  i.e.\ a saturated subset of $\Scott(\mathcal{S})$ which is a filter
  with respect to $\cov$, and we have $\mathcal{U} = \sat^{D} \bigcup_{A \in \mathcal{U}} \wayabove A$;

  \item if $\left( A_{i} \right)_{i \in I}$ is a family of elements
  of $\Fin{S}$ such that $\bigcup_{i \in I} \wayabove A_{i}$ is
  a filter with respect to $\cov$, then $\bigcup_{i \in I} \wayabove A_i
  = \bigcup \left\{ \wayabove A \mid A \cov^{D} \bigcup_{i \in I}
\wayabove A_i \right\}$.
  \end{enumerate}
\end{proposition}
\begin{proof}
  For the first claim, let $\mathcal{U} \in \Sat{\mathcal{S}^{D}}$.
  To see that $\bigcup_{A \in \mathcal{U}} \wayabove A$ is a Scott open
   filter, it suffices to show that 
  \[
    \bigcup_{A \in \mathcal{U}} \wayabove A
    = \bigcup_{\mathcal{V} \in \Fin{\mathcal{U}}}
    \;
    {\bigvee_{\mathclap{B \in \mathcal{V}}}}^{F} \wayabove B
    \]
  because the right hand side is a directed union of Scott open
  filters, and hence it is a Scott open filter.
  But this follows from the axiom \ref{D4'} and the fact that
  $\mathcal{U}$ is saturated. Then, the equation 
  $\mathcal{U} = \sat^{D} \bigcup_{A \in \mathcal{U}} \wayabove A$
  follows from \ref{D3'} and \ref{D4'}.
  The second claim follows from Lemma \ref{lem:DeGrootSK}.
\end{proof}

We show that $\Patch(\mathcal{S})$
is the patch of the de Groot dual $\mathcal{S}^{d}$.
First, there is a formal topology map $\varepsilon_{d} \colon
\Patch(\mathcal{S}) \to \mathcal{S}^{d}$ defined on
the generators of $T^{d}$ by 
\[
  \elPT{A} \mathrel{\varepsilon_{d}} \left\{ \lcut(A) \right\}
  \defeqiv
  \elPT{A} \cov_{P} \left\{ \lcut(A) \right\}.
\]
\begin{lemma} \label{lem:edFTopMap}
  \leavevmode
  \begin{enumerate}
    \item\label{lem:edFTopMap1} $\varepsilon_{d}$ is indeed a formal topology map.
    \item\label{lem:edFTopMap2} $\varepsilon_{d}$ is a perfect map.
    \item\label{lem:edFTopMap3} $\varepsilon_{d}$ is a monomorphism in $\FTop$.
  \end{enumerate}
\end{lemma}
\begin{proof}
  \ref{lem:edFTopMap1}.
  It suffices to show that the axiom \ref{D4}
  is derivable in $T_{P}$. Suppose that
$\wayabove A \subseteq \bigvee^{F}\! \left\{
  \wayabove A_{0}, \dots, \wayabove A_{n-1} \right\}$.
  By \ref{l3}, we have
$\left\{ \lcut(A) \right\}
\cov_{P}\
\left\{ \left\{ \lcut(B) \right\} \mid A \ll B \right\}$.
Let $B \gg A$. Then, there exist $B_{0} \gg A_{0}, \dots,
B_{n-1} \gg A_{n-1}$ such that $B_{0} \downarrow \cdots   \downarrow
B_{n-1} \ll B$. For each $i < n$, we have
  $
  \Fin{P_{P}} \cov_{P} \left\{ \left\{ \lcut(A_{i}) \right\} \right\}
  \cup \left\{ \left\{ \rcut(b) \right\} \mid b \in B_{i}\right\}
  $
by Lemma \ref{lem:LittleFact}. Then, by \ref{r2} we have
\[
  \Fin{P_{P}} \cov_{P} \left\{ \left\{ \lcut(A_{i}) \right\} \mid i <
n \right\}
  \cup \left\{ \left\{ \rcut(c) \right\} \mid c \in  B_{0} \downarrow
  \cdots \downarrow  B_{n-1} \right\}.
\]
For each $c \in B_{0} \downarrow \cdots \downarrow  B_{n-1}$, we have
$\left\{ \rcut(c), \lcut(B) \right\} \cov_{P} \emptyset$ 
by \ref{l2} and \ref{D}.
Hence,
  $
  \left\{ \lcut(B) \right\} \cov_{P}
  \left\{ \left\{ \lcut(A_{i}) \right\} \mid i < n \right\}.
  $
Therefore $\left\{ \lcut(A) \right\} \cov_{P}
\left\{ \left\{ \lcut(A_{i}) \right\} \mid i < n \right\}$.

\medskip
\noindent
  \ref{lem:edFTopMap2}.  It suffices to show that $B \gg A$
  implies $\left\{\lcut(B) \right\} \ll \left\{\lcut(A) \right\}$ in
  $\Patch(\mathcal{S})$. Suppose that $B \gg A$. Since
  $\mathcal{S}$ is compact and locally compact, there exists $ C \in
  \Fin{S}$ such that $S \cov C$ and $c \ll S$ for each $c \in C$.
  By the axioms \ref{r1} and \ref{r3}, we have
  $\left\{  \lcut(B)\right\}
  \cov_{P}
  \left\{ \left\{ \lcut(B), \rcut(c) \right\} \mid c \in C \right\}
   \subseteq
   \wb_{P}(\left\{ \lcut(A) \right\})$. 
   Hence,
  $\left\{\lcut(B) \right\} \ll \left\{\lcut(A) \right\}$.

\medskip
\noindent
\ref{lem:edFTopMap3}.\ The proof is similar to that of Lemma
\ref{lem:Core}.\ref{lem:CorefMono}.
\end{proof}

\begin{proposition}
  For any compact regular formal topology $\mathcal{S}'$ and a perfect
  formal topology map $r \colon \mathcal{S}' \to \mathcal{S}^{d}$, there exists a
  unique
  formal topology map $\tilde{r} \colon \mathcal{S}' \to
  \Patch(\mathcal{S})$ such that $\varepsilon_{d} \circ
  \tilde{r} = r$.
\end{proposition}
\begin{proof}
  Let $r \colon \mathcal{S}' \to \mathcal{S}^{d}$ be a perfect map from a
  compact regular formal topology $\mathcal{S}'$. Define
  $\tilde{r} \colon \mathcal{S}' \to \Patch(\mathcal{S})$ on the
  generators of $T^{d}$ by
  \begin{align*}
    \begin{aligned}
    b \mathrel{\tilde{r}} \left\{ \rcut(a) \right\}
    &\defeqiv \left( \exists a' \ll a \right) b \in
    \left( r^{-}\left\{ \lcut(\left\{ a' \right\}) \right\} \right)^{*},\\
    b \mathrel{\tilde{r}} \left\{ \lcut(A) \right\}
    &\defeqiv b \mathrel{r} \left\{ \lcut(A) \right\}.
    \end{aligned}
  \end{align*}
  We must show that $\tilde{r}$ respects the axioms of $T^{d}$.
  Since $r$ is a formal topology map from $\mathcal{S}'$ to
  $\mathcal{S}^{d}$,  $\tilde{r}$ respects
  the axioms \ref{l1} -- \ref{l3}. To see that $\tilde{r}$ respects the
  other axioms, we first show that
  \begin{equation}\label{eq:SOF}
    A \ll B \implies S'
    \cov'
    \bigcup_{b \in B}\tilde{r}^{-}\left\{
    \rcut(b) \right\} \cup \tilde{r}^{-}\left\{ \lcut(A) \right\}
  \end{equation}
  for all $A, B \in \Fin{S}$. Suppose that $A \ll B$. Then, 
  there exists $C = \left\{ c_{0},\dots,c_{n-1} \right\}$
  and $D = \left\{ d_{0},\dots, d_{n-1} \right\}$ such that
  $A \cov C$ and for each $i < n$, there exists $b \in B$
  such that $c_{i} \ll  d_{i} \ll b$.
  Since $r \colon \mathcal{S}' \to \mathcal{S}^{d}$ is perfect and
  $\mathcal{S}'$ is regular, we have
    $
    r^{-} \left\{ \lcut(\left\{ d_{i} \right\}) \right\}
    \lll r^{-} \left\{ \lcut(\left\{ c_{i} \right\}) \right\}
    $
  for each $i < n$. Hence, 
  \begin{align*}
    S'
    &\cov'
      \bigcup_{d \in D} \left( r^{-}\left\{ \lcut(\left\{ d \right\})
    \right\} \right)^{*}
      \cup \left( r^{-}\left\{ \lcut(\left\{ c_{0} \right\}) \right\}
      \downarrow \cdots \downarrow
      r^{-}\left\{ \lcut(\left\{ c_{n-1} \right\}) \right\}
      \right) \\
    &\cov'
      \bigcup_{d \in D} \left( r^{-}\left\{ \lcut(\left\{ d \right\})
    \right\} \right)^{*}
      \cup r^{-}\left( \left\{ \lcut(\left\{ c_{0} \right\}) \right\}
      \downarrow \cdots \downarrow
      \left\{ \lcut(\left\{ c_{n-1} \right\}) \right\}
      \right) \\
    &\cov'
      \bigcup_{d \in D} \left( r^{-}\left\{ \lcut(\left\{ d \right\})
    \right\} \right)^{*}
      \cup r^{-}\left\{ \lcut(C) \right\}\\
    &\cov'
    \bigcup_{b \in B}\tilde{r}^{-}\left\{
    \rcut(b) \right\} \cup \tilde{r}^{-}\left\{ \lcut(A) \right\}.
  \end{align*}
  We now show that $\tilde{r}$ respects the remaining axioms.

  \medskip
  \noindent\ref{r1} Since $\mathcal{S}$ is locally compact and compact, there
  exists $A \ll S$ such that $S \cov A$. Since $\left\{
  \lcut(A) \right\} \cov^{d} \emptyset$ by \ref{D4}, we have
  $S' \cov' \bigcup_{b \in S}\tilde{r}^{-}\left\{ \rcut(b) \right\}$
  by \eqref{eq:SOF}.

  \medskip
  \noindent\ref{r2}
  Suppose that $d \in \tilde{r}^{-}\left\{ \rcut(a) \right\}
  \downarrow \tilde{r}^{-}\left\{ \rcut(b) \right\}$. Then, there
  exist $a' \ll a$ and $b' \ll b$ such that
  $d \downarrow r^{-}\left\{ \lcut(\left\{ a' \right\}) \right\} \cov' \emptyset$ and 
  $d \downarrow r^{-}\left\{ \lcut(\left\{ b' \right\}) \right\} \cov' \emptyset$.
  Since $\mathcal{S}$ is stably compact, there exists $C \in \Fin{S}$ such that
  $\left\{ \lcut(C) \right\} \cov^{d} \left\{ \left\{
    \lcut(\left\{ a' \right\})\right\}, \left\{\lcut(\left\{ b' \right\})
  \right\}\right\}$ and $C \ll a \downarrow b$.
  Then by \eqref{eq:SOF}, we have
  \begin{align*}
    d
    &\cov' 
   \left( \bigcup_{c \in a \downarrow b}\tilde{r}^{-}\left\{ \rcut(c)
   \right\}
   \cup \tilde{r}^{-}\left\{ \lcut(C) \right\} \right) \downarrow d \\
    &\cov' 
    \bigcup_{c \in a \downarrow b}\tilde{r}^{-}\left\{ \rcut(c) \right\}
    \cup \left( \tilde{r}^{-}\left\{ \lcut(C) \right\} \downarrow d \right)\\
    &\cov' 
    \bigcup_{c \in a \downarrow b}\tilde{r}^{-}\left\{ \rcut(c) \right\}
    \cup  \left( \tilde{r}^{-}\left\{ \left\{ \lcut(\left\{ a'
    \right\}) \right\},
    \left\{ \lcut(\left\{ b' \right\}) \right\} \right\} \downarrow d \right)\\
    &\cov' 
    \bigcup_{c \in a \downarrow b}\tilde{r}^{-}\left\{ \rcut(c) \right\}.
  \end{align*}
  Verification of the axioms \ref{r3}, \ref{r4}, and \ref{Loc} is straightforward.

  \medskip
  \noindent\ref{D} Suppose that
  $b \in \tilde{r}^{-} \left\{ \lcut(\left\{ a \right\}) \right\}
  \downarrow \tilde{r}^{-} \left\{ \rcut( a ) \right\}$. Then, there
  exists $a' \ll a$ such that $b \downarrow r^{-}\left\{
    \lcut(\left\{ a' \right\}) \right\} \cov' \emptyset$. Since $r$ is perfect and $\mathcal{S}'$
  is regular, we have
  \begin{align*}
    b
    &\cov' \left( \left(r^{-} \left\{ \lcut(\left\{ a \right\})
    \right\} \right)^{*} \cup r^{-}\left\{ \lcut({a'}) \right\}\right)
    \downarrow b \\
    &\cov' \left( \left(r^{-} \left\{ \lcut(\left\{ a \right\})
    \right\} \right)^{*}\downarrow b \right) \cup \left(
    r^{-}\left\{ \lcut({a'}) \right\} \downarrow b\right) \\
    &\cov'  \left(r^{-} \left\{ \lcut(\left\{ a \right\})
  \right\} \right)^{*}\downarrow r^{-}\left\{ \lcut(\left\{ a
  \right\}) \right\}\\
    &\cov'  \emptyset.
  \end{align*}
  Hence, $\tilde{r}$ gives rise to a formal topology map from
  $\mathcal{S}'$ to $\Patch(\mathcal{S})$, which clearly satisfies
  $\varepsilon_{d} \circ \tilde{r} = r$. The uniqueness of $\tilde{r}$
  follows from Lemma \ref{lem:edFTopMap}.\ref{lem:edFTopMap3}.
\end{proof}

Thus, $\varepsilon_{d} \colon \Patch(\mathcal{S}) \to \mathcal{S}^{d}$
satisfies the universal property of the patch of $\mathcal{S}^{d}$.
\begin{theorem}\label{thm:PatchdeGroot}
  For any stably compact formal topology $\mathcal{S}$, we have
  $\Patch(\mathcal{S}) \cong \Patch(\mathcal{S}^{d})$.
\end{theorem}

\subsection{Spectral formal topologies}\label{sec:Spectral}
\begin{definition}
  A formal topology $\mathcal{S}$ is \emph{spectral} if the base of
  $\mathcal{S}$ is equipped with a meet semilattice structure $(S, 1,
  \wedge)$ that is compatible with the cover of $\mathcal{S}$ and
  such that $a \ll a$ for each $a \in S$.
  A formal topology map $r \colon \mathcal{S} \to \mathcal{S}'$
  between spectral formal topologies is \emph{spectral} if
  for each $a' \in S'$ there exists $A \in \Fin{S}$ such that
$
A =_{\mathcal{S}} r^{-}\left\{ a' \right\}.
$
A formal topology $\mathcal{S}$ is \emph{Stone} if it is spectral and $a \lll a$
  for each $a \in S$.
\end{definition}
Note that a spectral formal topology $\mathcal{S}$ is Stone if and only if 
it is regular if and only if every element $a \in S$ is complemented.
The spectral formal topologies and spectral maps form a category that is dually
equivalent to the category of distributive lattices and lattice homomorphisms
\cite[Chapter II, Section 3]{johnstone-82}.
Thus, its full subcategory of Stone topologies is dually equivalent to
the category of Boolean algebras and homomorphisms.

The following lemma prepares for Theorem \ref{thm:CorefPatchSpec} below.
\begin{lemma}\label{lem:Spec}
  \leavevmode
  \begin{enumerate}
    \item 
    \label{lem:SpectSLK}
    Every spectral formal topology is stably compact.

    \item 
    \label{lem:SpectPerfect}
    A formal topology map between spectral formal topologies
    is perfect if and only if it is spectral.

    \item \label{lem:PatchSpect}
     If $\mathcal{S}$ is a spectral formal topology, then $\Patch(S)$
     is Stone.
  \end{enumerate}
\end{lemma}
\begin{proof}
\ref{lem:SpectSLK} and \noindent\ref{lem:SpectPerfect} are
straightforward. As for \ref{lem:PatchSpect}, 
  if $\mathcal{S}$ is a spectral formal topology,
  then the axioms \ref{r1}, \ref{r2}, and \ref{l3} are equivalent to
  the following axioms:
  \begin{enumerate}
    \item[(R1)'] $\top \vdash \rcut(1)$
    \item[(R2)'] $\rcut(a) \wedge \rcut(b) \vdash \rcut(a \wedge b)$
    \item[(L3)'] $\lcut(A) \wedge \lcut(B) \vdash \lcut(A \cup B)$
  \end{enumerate}
  Moreover, the axiom \ref{r4} becomes trivial. Thus, the axioms
  of the theory $T_{P}$ only have finite joins. Since
the base of $\Patch(\mathcal{S})$ is equipped with a finite meet
structure with $1 \defeql \emptyset$ and
$\elPT{A} \wedge \elPT{B} \defeql
\elPT{A} \cup \elPT{B}$, it is a spectral formal topology. Since
$\Patch(\mathcal{S})$ is regular by Proposition \ref{prop:PatchLKReg}, it is Stone.
\end{proof}

\begin{theorem}[{\citet[Theorem
  3.3]{escardo2001regular}}]\label{thm:CorefPatchSpec}
  The patch construction exhibits the category of
  Stone formal topologies
  as a coreflective subcategory of the category
  of spectral formal topologies and spectral maps.
\end{theorem}
Hence $\Patch(\mathcal{S})$ represents the Boolean algebra generated
by the distributive lattice represented by 
the spectral topology $\mathcal{S}$; see
\citet{cederquist2000entailment}.

\section{Frames of perfect nuclei}\label{sec:PerfectNuclei}
\citet{escardo2001regular} defined the patch of a frame $X$ to be the
frame of perfect nuclei on $X$. We give a predicative presentation the
frame of perfect nuclei on a stably locally compact formal topology,
and clarify the relation between this topology and the geometric
presentation of the patch topology given in Section~\ref{sec:Patch}.

\subsection{Subtopologies}
Since the nuclei on a frame correspond to its sublocales 
\cite[Chapter II, Section 2.3]{johnstone-82}, we review
the notion of sublocale in the setting of formal topology.
\begin{definition}
  A \emph{subtopology} of a formal topology $\mathcal{S} = (S, \cov,
  \leq)$ is a formal topology $\mathcal{T} = (S,
  \cov_{\mathcal{T}}, \leq)$ such
  that 
  $
  a \cov U \implies a \cov_{\mathcal{T}} U
  $
  for all $a \in S$ and $U \subseteq S$.
  If $\mathcal{T}$ is a subtopology of $\mathcal{S}$, we write
  $\mathcal{T} \sqsubseteq \mathcal{S}$. In this case,
  there is a canonical embedding
  $\iota_{\mathcal{T}} \colon \mathcal{T} \to \mathcal{S}$ represented
  by the identity relation $\id_{S}$ on $S$.
\end{definition}

Let $\mathcal{S}$ be a formal topology and let $V \subseteq S$.
The \emph{open subtopology} of $\mathcal{S}$ determined by $V$
has a cover defined by
$
a \cov_V U \defeqiv a \downarrow V \cov U.
$
The \emph{closed subtopology} of $\mathcal{S}$ determined by $V$
has a cover defined by
  $
  a \cov^{\mathcal{S} - V} U \defeqiv a \cov V \cup U.
  $
If $\mathcal{S}$ is inductively generated, then the
open subtopology determined by $V$ can be characterised by
the axiom $S \cov_{V} V$, i.e.\ $a \cov_{V} V$ for
each $a \in S$. Similarly, the closed subtopology
determined by $V$ can be characterised by the axiom $V
\cov^{\mathcal{S} - V}
\emptyset$.

\subsubsection*{The frame of subtopologies}
In \cite[Section 4]{SublocFTop}, Vickers showed that the class
$\SubTop(\mathcal{S})$ of inductively generated subtopologies of an inductively generated
formal topology $\mathcal{S}$ forms a coframe: they are closed under
finite joins and set-indexed meets, and finite joins distribute over
set-indexed meets. 
Specifically, the
subtopology $\mathcal{S}_{\bot}$ with the trivial cover $\cov_{\bot}
=
S \times \Pow{S}$ is the least subtopology of $\mathcal{S}$. It is
characterised by the axiom $S \cov_{\bot} \emptyset$. If $\mathcal{S}_{0}$ and
$\mathcal{S}_{1}$ are inductively generated subtopologies of
$\mathcal{S}$, then their binary join $\mathcal{S}_{0} \vee
\mathcal{S}_{1}$ is characterised by the axioms of the form
  $
  a \downarrow b \cov_{\vee}  U \cup V
  $
where $a \cov_{0} U$ and $b \cov_{1} V$ are axioms of
$\mathcal{S}_{0}$ and $\mathcal{S}_{1}$ respectively. If $( \mathcal{S}_{i} )_{i \in I}$ is a family of
inductively generated subtopologies of $\mathcal{S}$, then the meet
$\bigwedge_{i \in I} \mathcal{S}_{i}$ is characterised by the union of axioms of each
$\mathcal{S}_{i}$.
We write $\SubTop(\mathcal{S})^{\op}$ for the frame
obtained by reversing the order of $\SubTop(\mathcal{S})$.

\subsection{Perfect subtopologies}\label{sec:PerferctSubTop}
A perfect nucleus on a stably locally compact locale
corresponds to the following notion in formal topology.
See \citet[Definition 2.3]{escardo2001regular} for the localic notion.
\begin{definition}
  A subtopology $\mathcal{S}'$ of a stably locally compact formal
  topology $\mathcal{S}$ is $\emph{perfect}$ if the canonical
  embedding $\iota_{\mathcal{S}'} \colon \mathcal{S}' \to
  \mathcal{S}$ is perfect, i.e.\
  \[
    a \ll b \implies a \ll' b.
  \]
Note that every perfect subtopology is locally compact,
and hence inductively generated.
\end{definition}

In what follows, we fix a stably locally compact formal
topology $\mathcal{S}$. For each $a \in S$ and $A \in \Fin{S}$, we
introduce the following notations:
\begin{align*}
  \Closed{a} &\defeql \text{the closed subtopology of $\mathcal{S}$ determined by
  $\left\{ a \right\}$.} \\
  \KFit{A} &\defeql \bigwedge_{\mathclap{A \ll B}}\mathcal{S}_{B} \;\text{where
    $\mathcal{S}_{B}$ is the open subtopology of $\mathcal{S}$
    determined by $B \in \Fin{S}$.}
\end{align*}
Note that for each $U \subseteq S$, the closed
subtopology  of $\mathcal{S}$  determined by $U$ is $\bigwedge_{a \in U} \Closed{a}$. In
particular, we have $\bigwedge_{a \in S} \Closed{a} =
\mathcal{S}_{\bot}$. Note also that  $\KFit{\emptyset} =
\mathcal{S}_{\bot}$.

We show that the perfect subtopologies of $\mathcal{S}$ form a
set-based subframe of $\SubTop(\mathcal{S})^{\op}$ generated by the
subtopologies of the form $\Closed{a} \vee \KFit{A}$.
\begin{lemma}\label{lem:PerfctSubLoc}
  \leavevmode
  \begin{enumerate}
    \item\label{lem:PerfctSubLoc1} $\Closed{a} \vee \Closed{b} = \bigwedge_{ c \in a \downarrow
      b} \Closed{c}$.
    \item\label{lem:PerfctSubLoc2} $\KFit{A} \vee \KFit{B} = \KFit{A \cup B}$.
  \end{enumerate}
\end{lemma}
\begin{proof}
  \ref{lem:PerfctSubLoc1}. The axioms of $\Closed{a} \vee \Closed{b}$
  are of the form
        $
        c \cov' \emptyset
        $
      for each $c \in a \downarrow b$. Hence,
      $\Closed{a} \vee \Closed{b} = \bigwedge_{ c \in a \downarrow
      b} \Closed{c}$.

      \medskip
  \noindent\ref{lem:PerfctSubLoc2}.
  The axioms of $\KFit{A} \vee \KFit{B}$ are of the form $a \cov' A'
  \cup B'$ where $A \ll A'$ and $B \ll B'$, which is an axiom of 
  $\KFit{A \cup B}$. Conversely, every axiom of $\KFit{A \cup B}$
  is an axiom of $\KFit{A} \vee  \KFit{B}$. Thus, 
  $\KFit{A} \vee \KFit{B} = \KFit{A \cup B}$.
\end{proof}

\begin{lemma}\label{lem:PerfctSubLocBase}
  \leavevmode
  \begin{enumerate}
    \item\label{lem:PerfctSubLocBase1}
      $\Closed{a} \vee \KFit{A}$ is a perfect subtopology for each $a \in S$ and
      $A \in \Fin{S}$.
    \item\label{lem:PerfctSubLocBase2}
      The subtopology meet of a set-indexed family $\left(\Closed{a_{i}} \vee
      \KFit{A_{i}}  \right)_{i \in I}$ is perfect.
    \item\label{lem:PerfctSubLocBase3}
      Every perfect subtopology of $\mathcal{S}$ is a meet
      of subtopologies of the form $\Closed{a} \vee \KFit{A}$.
    \item\label{lem:PerfctSubLocBase4}
      The class of perfect subtopologies of $\mathcal{S}$ is closed under finite joins
      and set-indexed meets.
  \end{enumerate}
\end{lemma}
\begin{proof}
  \ref{lem:PerfctSubLocBase1}.
  Let $\cov'$ be the cover of $\Closed{a} \vee \KFit{A}$.
  It suffices to show that
  \[
    a \cov' U \implies \left( \forall b \ll a \right) \left( \exists
    B \in \Fin{U} \right) a \cov' B
  \]
  for all $a \in S$ and $U \subseteq S$. This is proved by induction on
  $\cov'$. Fix $U \subseteq S$, and put
  \[
    \Phi(a) \equiv \left( \forall b \ll a \right) \left( \exists
    B \in \Fin{U} \right) a \cov' B.
  \]
  The (reflex) and ($\leq$)-rules are straightforward. For
  (infinity)-rule, we must check each localised axiom of $\Closed{a} \vee
  \KFit{A}$, which is of the following form:
  \[
    c \cov' C \downarrow c  \tag*{$(c \leq a \amp A \ll C \in
      \Fin{S})$}
  \]
  Assume $\Phi(c')$ for all $c' \in C \downarrow c$.
   Let $b \ll c$. There exists $C' \in \Fin{S}$
  such that $A \ll C' \ll C$. Since $\mathcal{S}$ is stably locally
  compact, we have $C' \downarrow b \ll C \downarrow c$.
  Then, there exists $\left\{ b_{0},\dots,b_{n-1} \right\}$ such that
  $C' \downarrow b \cov \left\{ b_{i} \mid i < n \right\}$
  and for each $i < n$ there exists $c' \in C \downarrow c$ such that 
  $b_{i} \ll c'$. Thus, for
  each $i < n$  there exists  $B_{i} \in \Fin{U}$ such that
  $b_{i} \cov' B_{i}$. Since $c \cov' C' \downarrow c$ is an axiom
  of $\Closed{a} \vee \KFit{A}$, we have
  $b \cov' C' \downarrow b \cov' \bigcup_{i < n }
  B_{i}$. Hence $\Phi(c)$.

  \medskip
  \noindent\ref{lem:PerfctSubLocBase2}. Similar to \ref{lem:PerfctSubLocBase1}.

  \medskip
  \noindent\ref{lem:PerfctSubLocBase3}.
  Let $\mathcal{S}'$ be a perfect subtopology of $\mathcal{S}$.
  We show that
    $
    \mathcal{S}' = \bigwedge_{a \cov' A}\left(\Closed{a} \vee
    \KFit{A}\right).
    $
  Write $\mathcal{S}^{*}$ for $\bigwedge_{a \cov' A}\left(\Closed{a} \vee
  \KFit{A}\right)$. Since $\ll \mathbin{\subseteq} \ll'$,
  the topology $\mathcal{S}'$ satisfies every axiom of $\mathcal{S}^{*}$. Thus
  $\mathcal{S}' \sqsubseteq \mathcal{S}^{*}$. Conversely, suppose that
  $a \cov' U$. Let $b \ll a$. Since $\mathcal{S}'$ is perfect, we have
  $b \ll' a$. Thus there exists 
  $A, B \in \Fin{S}$ such that $A \ll B \subseteq U$ and $b \cov' A$.
  Then $b \cov^{*} B$, so $a \cov^{*} \wb(a) \cov^{*} U$.
  Hence $\mathcal{S}^{*} \sqsubseteq \mathcal{S}'$.
  
  \medskip
  \noindent\ref{lem:PerfctSubLocBase4}.
  The least subtopology $\mathcal{S}_{\bot}$ is
  clearly perfect. The general claim follow from
  \ref{lem:PerfctSubLocBase2}, \ref{lem:PerfctSubLocBase3}, Lemma
  \ref{lem:PerfctSubLoc}, and distributivity of binary joins over set-indexed meets.
\end{proof}

Thus, the perfect subtopologies of $\mathcal{S}$ form a set-based
subframe of $\SubTop(\mathcal{S})^{\op}$ generated by the subset
  $
  \left\{\Closed{a} \vee \KFit{A} \mid a \in S \amp A \in \Fin{S})
\right\}.
  $
Hence, we can present
this subframe as a formal topology $\Perf(\mathcal{S})$ as follows:
\begin{align}
  \begin{aligned}\label{def:FrmPNuclei}
    \Perf(\mathcal{S}) &\;\:= (S_{\mathfrak{P}}, \cov_{\mathfrak{P}}, \leq_{\mathfrak{P}}),\\
    S_{\mathfrak{P}} &\:\,\defeql S \times \Fin{S},\\
    (a,A) \leq_{\mathfrak{P}} (b,B) &\defeqiv 
   \Closed{b} \vee \KFit{B} \sqsubseteq \Closed{a} \vee \KFit{A},\\
   (a,A) \cov_{\mathfrak{P}} U  &\defeqiv 
  \bigwedge \left\{ \Closed{b} \vee \KFit{B} \mid  (b,B) \in U \right\}
  \sqsubseteq \Closed{a} \vee \KFit{A}.
  \end{aligned}
\end{align}
Note that  the use of the relation $\sqsubseteq$ in the definition of
$\Perf(\mathcal{S})$ is predicative since the inclusion relation $\sqsubseteq$
between inductively generated subtopologies can be described in terms
of axiom-sets.

\subsection{Equivalence of \texorpdfstring{$\Perf(\mathcal{S})$ and
$\Patch(\mathcal{S})$}{PSub(S) and Patch(S)}}
\label{sec:PerferctSubTopEquiv}
We show that the topology $\Perf(\mathcal{S})$ given by
\eqref{def:FrmPNuclei} is isomorphic to $\Patch(\mathcal{S})$.
First, we define a formal topology map $r \colon \Perf(\mathcal{S}) \to
\Patch(\mathcal{S})$, which is defined on the generators
of $\Patch(\mathcal{S})$ as follows:
\begin{align*}
  \begin{aligned}
    (a,A) \mathrel{r} \left\{ \rcut(b) \right\}
    &\defeqiv 
    (a,A) \cov_{\mathfrak{P}} (b, \emptyset),\\
    (a,A) \mathrel{r} \left\{ \lcut(B) \right\}
    &\defeqiv 
    (a,A) \cov_{\mathfrak{P}} \left\{ (c, B) \mid c \in S \right\}.
  \end{aligned}
\end{align*}
\begin{lemma}
  The relation $r$ determines a formal
  topology map from $\Perf(\mathcal{S})$ to $\Patch(\mathcal{S})$.
\end{lemma}
\begin{proof}
  We must show that $r$ respects the axioms of the theory $T_{P}$.

  \noindent\ref{r1}
    Note that $\bigwedge_{a \in S} \Closed{a} = \mathcal{S}_{\bot}$.

  \noindent\ref{r2}
  By Lemma \ref{lem:PerfctSubLoc}.\ref{lem:PerfctSubLoc1}.

  \noindent\ref{r3}
  If $a \cov \left\{ a_{0},\dots,a_{n-1} \right\}$, then
$\bigwedge_{i < n} \Closed{a_{i}} \sqsubseteq \Closed{a}$.

  \noindent\ref{r4}
  Similar to the above case, noting that $a \cov \left\{ b \in S
  \mid b \ll a \right\}$. 

  \noindent\ref{l1} Similar to the case \ref{r1} above.

  \noindent\ref{l2}
  If $A \cov B$, then $\KFit{A} \sqsubseteq \KFit{B}$.

  \noindent\ref{l3}
  First, note that $\KFit{A} = \bigwedge_{ A \ll B}\mathcal{S}_{B} =
  \bigwedge_{ A \ll B } \KFit{B}$. Thus, by 
  Lemma \ref{lem:PerfctSubLoc}.\ref{lem:PerfctSubLoc2}, we have
    $
    \KFit{A} \vee \KFit{B}
    = \KFit{A \cup B}
    = \bigwedge\left\{\KFit{C} \mid  A \cup B \ll C \right\}
    = \bigwedge\left\{\KFit{C} \mid A \ll C \amp B \ll C \right\}.
    $

  \noindent\ref{Loc}
  Suppose that $a \ll b$.
  We must show that $\KFit{ \left\{ a \right\}} \wedge
  \Closed{b} = \mathcal{S}_{\bot}$.
  Let $\cov'$ denotes the cover of
  $\KFit{ \left\{ a \right\}} \wedge
  \Closed{b}$.
  For any $c \in S$, we have $c \cov' b$ by the axiom of $\KFit{ \left\{ a \right\}}$ and
  $b \cov' \emptyset$ by the axiom of $\Closed{b}$. Hence 
  $\KFit{ \left\{ a \right\}} \wedge
  \Closed{b} = \mathcal{S}_{\bot}$.

  \noindent\ref{D}
  We must show that $\Closed{a} \vee \KFit{ \left\{ a
  \right\}} = \mathcal{S}$. But the axioms of $\Closed{a} \vee \KFit{
  \left\{ a \right\}} = \mathcal{S}$ are of the forms
  $a \cov' A$ where $a \ll A$,
  which is trivial. Thus, $\Closed{a} \vee \KFit{ \left\{ a
  \right\}} = \mathcal{S}$.
  \qedhere
\end{proof}

We define the inverse $s \colon \Patch(\mathcal{S}) \to
\Perf(\mathcal{S})$ of $r$ by 
\begin{equation} \label{def:Patch2SP}
  \elPT{A} \mathrel{s} (a,B) \defeqiv \elPT{A}
  \cov_{P}
  \left\{ \rcut(a), \lcut(B)\right\}
\end{equation}
where $\cov_{P}$ is the cover of $\Patch(\mathcal{S})$.
Before showing that $s$ is indeed a formal topology
map, we prove
the following auxiliary lemma.
\begin{lemma}\label{lem:AuxRelationS}
  For any $U \subseteq S_{\mathfrak{P}}$, $a \in S$ and $A \in \Fin{S}$,
  \[
    a \cov' A \implies 
      \left\{ \rcut(a), \lcut(A) \right\}
      \cov_{P}
      \left\{
      \left\{ \rcut(b), \lcut(B) \right\} \mid (b,B) \in U
      \right\}
  \]
  where $\cov'$ is the cover of  $\bigwedge_{(b,B) \in U}
  \Closed{b} \vee \KFit{B}$.
\end{lemma}
\begin{proof}
  The proof is by induction on $\cov'$.
  Fix $U \subseteq S_{\mathfrak{P}}$ and $A \in \Fin{S}$, and put
  \[
    \Phi(a) \equiv 
      \left\{ \rcut(a), \lcut(A) \right\}
      \cov_{P}
      \left\{
      \left\{ \rcut(b), \lcut(B) \right\} \mid (b,B) \in U
      \right\}.
  \]
  The case for (reflex)-rule follows
  from \ref{D}, and that of ($\leq$)-rule follows from \ref{r3}.
  As for (infinity)-rule, note that the axioms of 
  $\bigwedge_{(b,B) \in U} \Closed{b} \vee \KFit{B}$ are of the
  form $b \cov' C$ for each $(b,B) \in U$ and $B \ll C$.
  Localising it, we obtain $c \cov' C \downarrow c$
  where $(b,B) \in U \amp B \ll C \amp c \leq b$.
  Now assume $\Phi(c')$ for all $c' \in C \downarrow c$.
  Since $B \ll C$, we have
  $
  \Fin{P_{P}} \cov_{P} \left\{ \left\{ \lcut(B) \right\} \right\}
  \cup \left\{ \left\{ \rcut(c') \right\} \mid c' \in C \right\}
  $
  by Lemma \ref{lem:LittleFact}. Thus, 
  \begin{align*}
    \left\{ \rcut(c), \lcut(A)\right\}
    &\cov_{P} \left\{ \left\{ \lcut(B), \rcut(c), \lcut(A)\right\} \right\}
    \cup \left\{ \left\{ \rcut(c'), \rcut(c), \lcut(A)\right\}  \mid
  c' \in C \right\} \\
    &\cov_{P} \left\{ \left\{ \lcut(B), \rcut(b)\right\} \right\}
    \cup \left\{ \left\{ \rcut(c'), \rcut(c), \lcut(A)\right\}  \mid
  c' \in C \right\}&& \text{(by \ref{r3})}  \\
    &\cov_{P} \left\{ \left\{ \lcut(B), \rcut(b) \right\} \right\}
    \cup \left\{ \left\{ \rcut(c'), \lcut(A)\right\}  \mid
  c' \in C  \downarrow c \right\}&& \text{(by \ref{r2})} \\
    &\cov_{P} \left\{ \left\{ \rcut(b'), \lcut(B') \right\} \mid
  (b',B') \in U \right\}.&& \text{(by I.H)}
  \qedhere
  \end{align*}
\end{proof}

\begin{lemma}
  The relation $s$ given by \eqref{def:Patch2SP} is a formal topology
  map from $\Patch(\mathcal{S})$ to $\Perf(\mathcal{S})$.
\end{lemma}
\begin{proof}
  We must check the conditions \ref{FTM1}, \ref{FTM2}, and \eqref{FTM3}.
  The condition \ref{FTM1} follows from \ref{r1} and \ref{l1}, and
  \ref{FTM2} follows from \ref{r2} and \ref{l3}. 
  As for \eqref{FTM3}, suppose that $(a,A) \cov_{\mathfrak{P}} U$, i.e.\
      $\bigwedge_{(b,B) \in U}\left(\Closed{b} \vee \KFit{B}  \right) \sqsubseteq
      \Closed{a} \vee \KFit{A}$. This is equivalent to saying that 
       $\bigwedge_{(b,B) \in U} \Closed{b} \vee \KFit{B}$
       satisfies all the axioms of $\Closed{a} \vee \KFit{A}$.
       Since $a \cov' A'$ is an axiom of $\Closed{a} \vee \KFit{A}$
       for all $A' \gg A$, we have
       \begin{align*}
         s^{-}(a,A)
         \cov_{P} \left\{ \rcut(a), \lcut(A) \right\}
         &\cov_{P}
         \left\{\left\{ \rcut(a), \lcut(A') \right\}
            \mid A \ll A'  \right\}&& \text{(by \ref{l3})}\\
         &\cov_{P}
         \left\{ \left\{ \rcut(b), \lcut(B)\right\} 
            \mid (b,B) \in U \right\}&& \text{(by Lemma
            \ref{lem:AuxRelationS})}\\
         &\cov_{P} s^{-} U. &&\qedhere
       \end{align*}
\end{proof}

\begin{theorem}\label{thm:EqivEscardoOurs}
  The morphisms $r \colon \Perf(\mathcal{S}) \to
  \Patch(\mathcal{S})$ and $s \colon \Patch(\mathcal{S}) \to
  \Perf(\mathcal{S})$ are inverse to each other.
  Hence $\Perf(\mathcal{S})$ and $\Patch(\mathcal{S})$ are
  isomorphic.
\end{theorem}
\begin{proof}
  Straightforward from the definitions of $r$ and $s$.
\end{proof}

From the definitions of $r$ and $s$, it is clear that
$\Perf(\mathcal{S})$ and $\Patch(\mathcal{S})$ are essentially that same
presentations. The difference is that
$\Perf(\mathcal{S})$ is a presentation 
by a basis while $\Patch(\mathcal{S})$ is a presentation by a subbasis.

\section{Lawson topologies}\label{sec:LawsonTop}
We give a geometric characterisation of the Lawson topology
of a continuous lattice in terms of continuous basic cover.
Here, the Lawson topology of a continuous basic cover is presented by
a geometric theory whose models are the located subsets of the given
continuous basic cover. Classically, there is
another way of constructing the Lawson topology of a continuous
lattice as the patch of its Scott topology. We clarify the connection
between these two approaches in Section \ref{sec:LawsonasPatch}. We
also give a point-free account of the fundamental connection between
Lawson topologies of continuous lattices and the Vietoris monad on the
category of compact Hausdorff spaces; see Section \ref{sec:Vietoris}.

\subsection{Geometric theories of located subsets}
We fix a continuous basic cover
$\mathcal{S}$.
The following construction is motivated by the characterisation of
located subsets in Lemma
\ref{lem:CutAsScott}.
\begin{definition} \label{def:LocPower}
The \emph{Lawson
topology} of a continuous basic cover $\mathcal{S}$ is a formal topology
$\Located(\mathcal{S})$ presented by a geometric theory $T_{\Located}$
over the propositional symbols
\[
  P_{\Located} \defeql \left\{ \lcut(A) \mid A \in \Fin{S} \right\}
  \cup \left\{ \rcut(a) \mid a \in S \right\}
\] 
with the axioms \ref{r3} -- \ref{D} of the theory $T_{P}$ (see
Definition \ref{def:Patch}).
The models of $T_{\Located}$
bijectively correspond to the located subsets of $\mathcal{S}$.
\end{definition}

Given a continuous basic cover $\mathcal{S}$ and the geometric theory
$T_{\Located}$ as in Definition \ref{def:LocPower}, define a function
  $
  \wb_{\Located} \colon  \Fin{P_{\Located}} \to \Pow{\Fin{P_{\Located}}}
  $
by induction on $\Fin{P_{\Located}}$:
\begin{align*}
  \wb_{\Located}(\emptyset)
  &\defeql \left\{  \emptyset \right\},\\
  \wb_{\Located}(\elPT{A} \cup \left\{ \rcut(a) \right\})
  &\defeql
  \left\{ \elPT{B} \cup \left\{ \rcut(b) \right\} \mid
  \elPT{B} \in \wb_{\Located}(\elPT{A}) \amp  b \ll a \right\},\\
  \wb_{\Located}(\elPT{A} \cup \left\{ \lcut(A) \right\})
  &\defeql
  \left\{ \elPT{B} \cup \left\{ \lcut(B) \right\} \mid
  \elPT{B} \in \wb_{\Located}(\elPT{A}) \amp  A \ll B \right\}.
\end{align*}

\begin{lemma}\label{lem:wbL} \leavevmode
  \begin{enumerate}
    \item\label{lem:wbL1} $\elPT{B} \in \wb_{\Located}(\elPT{A}) \implies
      \elPT{B} \ll \elPT{A} \amp \elPT{B} \lll \elPT{A}$,
    \item\label{lem:wbL3} $\elPT{A}
      \cov_{\Located(\mathcal{S})} \wb_{\Located}(\elPT{A})$.
  \end{enumerate}
\end{lemma}
\begin{proof}
  Similar to the proof of Lemma \ref{lem:wbT}.
\end{proof}

\begin{proposition}\label{prop:LocPowerKR}
  $\Located(\mathcal{S})$ is a compact regular formal topology.
\end{proposition}
\begin{proof}
$\Located(\mathcal{S})$ is locally compact regular by Lemma
\ref{lem:wbL}. Since $\emptyset$ is a top element and $\emptyset \ll
\emptyset$, it is also compact.
\end{proof}

There exists a basic cover map $\varepsilon_{\Located} \colon
\Located(\mathcal{S}) \to \mathcal{S}$ defined by 
\begin{equation}\label{def:CorefLocated}
  \elPT{A} \mathrel{\varepsilon_{\Located}} a \defeqiv \elPT{A}
  \cov_{\Located(\mathcal{S})} \left\{
  \rcut(a) \right\}
\end{equation}
which is perfect by the definition of $\wb_{\Located}$.
\begin{proposition}\label{prop:UnivPropLocated}
  For any compact regular formal topology $\mathcal{S}'$ and a
  perfect basic cover map $r \colon \mathcal{S}' \to
  \mathcal{S}$, there exists a
  unique formal topology map $\tilde{r} \colon \mathcal{S}' \to
  \Patch(\mathcal{S})$ such that $\varepsilon_{\Located} \circ
  \tilde{r} = r$.
\end{proposition}
\begin{proof}
  The proof is analogous to that of Proposition \ref{prop:UnivPropPatch}.
  In particular, the unique formal topology map $\tilde{r} \colon
  \mathcal{S}' \to \Patch(\mathcal{S})$ is defined by
  \eqref{def:uniqExt}.
\end{proof}

Let $\KReg$ be the full subcategory of $\FTop$ consisting of compact
regular formal topologies. Since compact regular formal topologies are
locally compact and morphisms between them are always perfect, we have a
forgetful functor $U \colon \KReg \to \LFCov$. Then, Proposition
\ref{prop:UnivPropLocated} implies the following.
\begin{theorem}\label{thm:LRegLFCov}
  The forgetful functor $U \colon \KReg \to \LFCov$ has a right
  adjoint. The right adjoint is exhibited by the construction of
  Lawson topologies.
\end{theorem}

\subsection{Lawson topologies as patch topologies}
\label{sec:LawsonasPatch}
The forgetful functor $U \colon \FTopi \to \BCovi$ from
the category of inductively generated formal topologies to that of
inductively generated basic covers has a right adjoint
\cite{ConvFTop}.  The monad on $\FTopi$ induced by this adjunction
corresponds to the lower powerlocale monad
\cite{Vickers95constructivepoints}. Thus, we call the right adjoint
the lower powerlocale construction.
We show that the construction of Lawson topologies can be decomposed
into the patch construction and the lower powerlocale construction.

\begin{definition}
The \emph{lower powerlocale} $\Lower{\mathcal{S}}$ of an inductively
generated  basic cover $\mathcal{S}$ is presented by the geometric
theory whose models are the splitting subsets of $\mathcal{S}$.
Specifically, if $(I,C)$ is an axiom-set of $\mathcal{S}$, then
$\Lower{\mathcal{S}}$ can be presented by a geometric theory over $S$
with the axioms of the form
\[
  a \vdash \bigvee \left\{ b \mid b \in C(i,a) \right\}
\]
for each $a \in S$ and $i \in I(a)$.
There is a basic cover map $\sigma_{L}
\colon \Lower{\mathcal{S}} \to \mathcal{S}$ given by  
\begin{equation}\label{eq:LowerCounit}
  A \mathrel{\sigma_{L}} a \defeqiv A \cov_{\Lower{\mathcal{S}}} \left\{ a \right\}
\end{equation}
with the following universal property: for any formal topology
$\mathcal{S}'$ and a basic cover map $r \colon \mathcal{S}' \to
\mathcal{S}$, there exists a unique formal topology map $\overline{r}
\colon \mathcal{S}' \to \Lower{\mathcal{S}}$ such that $\sigma_{L}
\circ \overline{r} = r$.
\end{definition}

Note that, if $\mathcal{S}$ is continuous basic cover,
then $\Lower{\mathcal{S}}$ can be presented by the geometric theory over the
generators $\left\{ \rcut(a) \mid a \in S \right\}$ together with
axioms \ref{r3} and \ref{r4} of $T_{\Located}$ (cf.\ Lemma
\ref{lem:SplittingInLK}).

\begin{lemma}\label{lem:LowerSLK}
  If $\mathcal{S}$ is a continuous basic cover,
  then $\Lower{\mathcal{S}}$  is stably compact.
\end{lemma}
Before proving Lemma \ref{lem:LowerSLK}, we reprove some well known results
about continuous lattices and locales in terms of basic covers and formal
topologies; see \citet[Chapter IIV, Section 4.6]{johnstone-82} for the
corresponding localic results.
\begin{lemma}\label{lem:Retract}
  \leavevmode
  \begin{enumerate}
    \item\label{lem:Retract1} A retract of a continuous basic cover in
      $\BCov$ is continuous.
    \item\label{lem:Retract2} Every continuous basic cover is a retract of
      a finitely basic cover, where a basic cover $\mathcal{S}$ is
      \emph{finitely} if $a \ll a$ for all $a \in S$.
    \item\label{lem:Retract3} A retract of a stably locally compact formal topology in
      $\FTop$ is stably locally compact.
    \item\label{lem:Retract4} A retract of a compact formal topology is compact.
  \end{enumerate}
\end{lemma}
\begin{proof}
  \ref{lem:Retract1}.
  Let  $r \colon \mathcal{S} \to
  \mathcal{S}'$ and $s \colon \mathcal{S}' \to \mathcal{S}$ be basic
  cover maps such that $r \circ s = \id_{\mathcal{S}'}$, and suppose
  that $\mathcal{S}$ is continuous. Define a function $\wb \colon S'
  \to \Pow{S'}$ by
  \begin{equation}\label{lem:Retract1wb}
    \wb(a) \defeql \left\{ b \in S' \mid \left( \exists c \in S \right) b
      \mathrel{s} c \amp c \ll r^{-}a \right\}.
    \end{equation}
  We show that $\wb$ makes $\mathcal{S}'$ continuous (see Definition
  \ref{def:LK}). First, since $\mathcal{S}$ is continuous, we have
    $
    a \cov' s^{-}r^{-}a \cov' s^{-} \left\{ c \in S \mid c \ll
      r^{-}a \right\} \cov' \wb(a).
    $
  Next, let $b \in \wb(a)$, and suppose that $a \cov' U$.
  Then, there exists $c \in S$ such that $b \mathrel{s} c$
  and $c \ll r^{-}a$. Since $r^{-}a \cov r^{-}U$, there exists
  $B \in \Fin{U}$ such that $c \cov r^{-}B$. Then, 
  $b \in s^{-}c \cov s^{-}r^{-}B \cov B$. Hence, $b \ll' a$.

  \medskip
  \noindent\ref{lem:Retract2}.
  If $\mathcal{S}$ is a continuous basic cover,
  then it is a retract of the finitely basic cover
  $\mathcal{S}_{F} = (S, \cov_{F})$ given by
    $
    a \cov_{F} U \defeqiv \left( \exists A \in \Fin{U} \right) a \cov A.
    $
  The retraction $r \colon \mathcal{S}_{F} \to \mathcal{S}$ is given
  by $r \defeql \ll$.

  \medskip
  \noindent\ref{lem:Retract3}.
  Let  $r \colon \mathcal{S} \to
  \mathcal{S}'$ and $s \colon \mathcal{S}' \to \mathcal{S}$ be formal
  topology maps such that $r \circ s = \id_{\mathcal{S}'}$, and suppose
  that $\mathcal{S}$ is stably locally compact. Then, 
  $\mathcal{S}'$ is locally compact by \ref{lem:Retract1}. Let
  $a \ll' a'$ and $b \ll' b'$, and suppose that $a' \downarrow b' \cov' U$.
  Let $\wb$ be the function defined by \eqref{lem:Retract1wb}.
  Then, there exist $\left\{ a_{0},\dots,a_{n-1} \right\} \in
  \Fin{\wb(a')}$ and $\left\{ b_{0},\dots,b_{m-1} \right\} \in
  \Fin{\wb(b')}$ such that $a \cov' \left\{ a_{i} \mid i < n\right\}$
  and $b \cov' \left\{ b_{j} \mid j < m\right\}$.
  Then, for each $i < n$ and $j < m$, there exist $c_{i}, c_{j} \in S$
  such that $a_{i} \mathrel{s} c_{i} \ll r^{-} a'$ and 
  $b_{j} \mathrel{s} c_{j} \ll r^{-} b'$. Since $\mathcal{S}$ is
  stably locally compact, we have 
  \[
    c_{i} \downarrow c_{j} \ll r^{-}a' \downarrow r^{-}b' \cov
    r^{-}(a' \downarrow b' ) \cov r^{-}U.
  \]
  Then, there exists $C_{i,j} \in \Fin{U}$ such that $c_{i} \downarrow
  c_{j} \cov r^{-}C_{i,j}$. Hence, 
  \[
    a_{i} \downarrow b_{j} \cov'
    s^{-} c_{i} \downarrow s^{-} c_{j} \cov' s^{-}(c_{i} \downarrow
    c_{j}) \cov' s^{-}r^{-}C_{i,j} \cov' C_{i,j},
  \]
  so that $a \downarrow b \cov' \left\{ a_{i} \mid i < n \right\}
  \downarrow \left\{ b_{j} \mid j < m \right\} \cov'
  \bigcup_{\substack{i < n \\j < m}}C_{i,j}$. Thus, $a \downarrow
  b \ll' a' \downarrow b'$.

  \medskip
\noindent\ref{lem:Retract4}. Straightforward.
\end{proof}

\begin{proof}[{Proof of Lemma \ref{lem:LowerSLK}}]
  By Lemma \ref{lem:Retract}.\ref{lem:Retract2},
  $\mathcal{S}$ is a retract of the finitary basic cover
  $\mathcal{S}_{F}$.
  Since $\Lower{\mathcal{S}_{F}}$ is generated by the axioms of the
  form \ref{r3} only, $\Lower{\mathcal{S}_{F}}$ is a spectral formal topology,
  which is stably compact (cf.\ Lemma \ref{lem:Spec}).
  Hence, $\Lower{\mathcal{S}}$  is stably compact by Lemma
  \ref{lem:Retract}.\ref{lem:Retract3} and
  \ref{lem:Retract}.\ref{lem:Retract4}.
\end{proof}

Let $\mathcal{S}$ be a continuous basic cover.  The map
$\varepsilon_{\Located} \colon \Located(\mathcal{S}) \to \mathcal{S}$
given by \eqref{def:CorefLocated} uniquely extends to a formal
topology map $\overline{\varepsilon_{\Located}} \colon
\Located(\mathcal{S}) \to \Lower{\mathcal{S}}$ via $\sigma_{L} \colon
\Lower{\mathcal{S}} \to \mathcal{S}$. It is easy to show that both
$\overline{\varepsilon_{\Located}} \colon \Located(\mathcal{S}) \to
\Lower{\mathcal{S}}$ and $\sigma_{L} \colon \Lower{\mathcal{S}} \to
\mathcal{S}$ are perfect. 
Now, let $r \colon \mathcal{S}' \to \Lower{\mathcal{S}}$ be a perfect
formal topology map from a compact regular formal topology
$\mathcal{S}'$. By Proposition \ref{prop:UnivPropLocated}, the
composition $\sigma_{L} \circ r$ uniquely extends to a formal topology
map $s \colon \mathcal{S}' \to \Located(\mathcal{S})$ via
$\varepsilon_{\Located} \colon \Located(\mathcal{S}) \to \mathcal{S}$.
Then, $\sigma_{L} \circ r = \varepsilon_{\Located} \circ s =
\sigma_{L} \circ \overline{\varepsilon_{\Located}} \circ s$. Hence by
the universal property of $\sigma_{L}$, we have $r =
\overline{\varepsilon_{\Located}} \circ s$. Then, $s$ is a unique
formal topology map from $\mathcal{S}'$ to $\Located(\mathcal{S})$
such that $r = \overline{\varepsilon_{\Located}} \circ s$. The
following diagram summarises the argument.
\[
  \xymatrix{
    & \Located(\mathcal{S}) \ar[dl]_{\varepsilon_{\Located}}
    \ar[d]^-{\overline{\varepsilon_{\Located}}}
    & \ar[l]_-{s} \ar[ld]^-{r} \mathcal{S}' \\
   \mathcal{S}  &\Lower{\mathcal{S}} \ar[l]^{\sigma_{L}} & 
  }
\]
By Theorem \ref{thm:CorefPatchSLKC}, we have the following theorem.
\begin{theorem}\label{thm:LawsonTop}
  The Lawson topology $\Located(\mathcal{S})$
  is naturally isomorphic to $\Patch(\Lower{\mathcal{S}})$.
\end{theorem}

Classically, the Lawson topology of a continuous lattice agrees with
the patch topology on its Scott topology; see \citet[Lemma
V-5.15]{gierz2003continuous}. We give a constructive
account of this fact by showing that
the Scott topology of a continuous basic cover $\mathcal{S}$
is the de Groot dual of the lower powerlocale $\Lower{\mathcal{S}}$.
 
Recall that a \emph{dcpo} is a poset with all directed
joins, and a dcpo homomorphism is a function that preserves directed
joins. A \emph{preframe} is a dcpo with finite meets
that distribute over the directed joins, and a
preframe homomorphism is a function that preserves finite meets and
directed joins.

The next proposition, which is a consequence of the preframe coverage
theorem due to \citet{JohnstoneVickersPreframePresent}, is crucial for
our development. We give a predicative proof for the sake of
completeness.
\begin{proposition}\label{prop:preframeCoverage}
  Let $\mathcal{S}$ be an inductively generated basic cover.
  Then, for any preframe $X$ and a dcpo homomorphism $f \colon
  \Sat{\mathcal{S}} \to X$, there exists a unique preframe homomorphism
  $F \colon \Sat{\Lower{S}} \to X$ such that
  $F \circ \sigma_{L}^{*} = f$, where
   $\sigma_{L}^{*} \colon \Sat{\mathcal{S}} \to
  \Sat{\Lower{\mathcal{S}}}$ is the suplattice homomorphism induced
  by the basic cover map $\sigma_L \colon \Lower{\mathcal{S}} \to
  \mathcal{S}$ given by \eqref{eq:LowerCounit}.
\end{proposition}
\begin{proof}
  First, we define a function $(-)^{*} \colon \Fin{\Fin{S}} \to
  \Fin{\Fin{S}}$ by induction:
  \begin{align*}
    \emptyset^{*} &\defeql \left\{ \emptyset \right\},\\
    \left( \mathcal{V} \cup \left\{ A \right\} \right)^{*} &\defeql
    \left\{ B \cup C \mid B \in \mathcal{V}^{*} \amp C \in
    \PFin{A}\right\},
  \end{align*}
  where $\PFin{A}$ is the set of inhabited finitely enumerable
  subsets of $A$. Note that for each $\mathcal{V} \in
  \Fin{\Fin{S}}$ we have
    $
    \bigvee_{A \in \mathcal{V}} \medwedge_{a \in A}\sat_{L}\left\{ a \right\}
    =_{\Lower{\mathcal{S}}} 
    \bigwedge_{B \in \mathcal{V}^{*}} \medvee_{b \in B}\sat_{L}\left\{
    b \right\},
    $
    where $\sat_{L}$ is the saturation operation of
    $\Lower{\mathcal{S}}$.

  Let $f \colon \Sat{\mathcal{S}} \to X$ be a dcpo homomorphism to a
  preframe $X$. Define a function
  $F \colon \Sat{\Lower{S}} \to X$ by
  \begin{align*}
    F(\sat_{L}\mathcal{U}) 
    &\defeql\;
    \bigvee_{\mathclap{\mathcal{V} \in \Fin{\mathcal{U}}}}
    \widetilde{f}(\mathcal{V})
  \end{align*}
  for each $\mathcal{U} \subseteq \Fin{S}$, where
  $\widetilde{f}(\mathcal{V}) \defeql \bigwedge_{B \in \mathcal{V}^{*}}
  f(\sat B)$.

  We must show that
  $F$ is well defined, i.e.\ 
  $F(\sat_{L} \mathcal{U}) 
  = F(\sat_{L} \mathcal{U}')$ whenever
  $\sat_{L} \mathcal{U} = \sat_{L} \mathcal{U}'$.
  To this end, it suffices to show that
  \begin{equation}\label{eq:FwellDefined}
    \mathcal{V} \cov_{\Lower{\mathcal{S}}} \mathcal{U} \implies
    \widetilde{f}(\mathcal{V}) \leq_{X} F(\mathcal{U})
  \end{equation}
  for all $\mathcal{V} \in \Fin{\Fin{S}}$ and $\mathcal{U} \subseteq
  \Fin{S}$.  Fix $\mathcal{U} \subseteq \Fin{S}$, and define a
  predicate $\Phi$ on $\Fin{\Fin{S}}$ by
  \[
    \Phi(\mathcal{W}) \equiv \left( \forall \mathcal{V} \in \Fin{\Fin{S}}\right)
    \widetilde{f}(\mathcal{V} \cup \mathcal{W}) \leq_{X}
    F(\mathcal{V} \cup \mathcal{U}).
  \]
  By induction on $\mathcal{W}$, it is easy to see that
    $
      \left( \forall A \in \mathcal{W} \right) \Phi(\left\{ A
    \right\})
  \implies \Phi(\mathcal{W}).
    $
  Thus, to prove \eqref{eq:FwellDefined}, it suffices to show that
  \[
    A \cov_{\Lower{\mathcal{S}}} \mathcal{U} \implies \Phi(\left\{ A
    \right\})
  \]
  for all $A \in \Fin{S}$.
  This is proved by induction on
  $\cov_{\Lower{\mathcal{S}}}$. The cases for (reflex) and
  ($\leq$)-rules are straightforward. For (infinity)-rule, let
  $a \in S$, $i \in I(a)$ and $A \in \Fin{S}$, and suppose that
  $\Phi(\left\{ A \cup \left\{ b \right\} \right\})$ for all
  $b \in C(a,i)$. Let $\mathcal{V} \in \Fin{\Fin{S}}$. Note that
  each element $B \in \widetilde{f}(\mathcal{V} \cup \left\{ A \cup \left\{
  a \right\} \right\})$ is either of the following forms:
  \begin{enumerate}
    \item $B = B_{0} \cup B_{1} \cup \left\{ a \right\}$, where
      $B_{0} \in \mathcal{V}^{*}$ and $B_{1} \in \Fin{A}$,
    \item $B = B_{0} \cup B_{1}$, where
      $B_{0} \in \mathcal{V}^{*}$ and $B_{1} \in \PFin{A}$.
  \end{enumerate}
  If $B$ is of the form $B_{0} \cup B_{1} \cup
  \left\{ a \right\}$ for some $B_{0} \in \mathcal{V}^{*}$ and $B_{1}
  \in \Fin{A}$, then
  \[
    f(\sat B)
    \leq_{X}
    \bigvee\left\{ f(\sat \left(B_{0} \cup B_{1} \cup C
    \right)) \mid C \in \Fin{C(a,i)}\right\}
  \]
  because $f$ preserves directed joins. Since $X$ is a preframe, we
  have
  \begin{align*}
    \widetilde{f}(\mathcal{V} \cup \left\{ A \cup \left\{ a
    \right\} \right\})
    &= 
    z \wedge \;\; \medwedge_{\mathclap{\substack{
        B_{0} \in \mathcal{V}^{*} \\
        B_{1} \in \Fin{A}}}}\;
      f(\sat \left(B_{0} \cup B_{1} \cup \left\{ a \right\} \right))\\
    &\leq_{X}
      \bigvee\Bigl\{ z \wedge \;\;
        \medwedge_{\mathclap{\substack{
        B_{0} \in \mathcal{V}^{*} \\
        B_{1} \in \Fin{A}}}}\;
        f(\sat \left(B_{0} \cup B_{1} \cup C
        \right)) \mid C \in \Fin{C(a,i)}\Bigr\},
      \end{align*}
      where $z \defeql \medwedge
      \left\{ f(\sat \left(B_{0} \cup B_{1}\right)) \mid 
        B_{0} \in \mathcal{V}^{*} \amp B_{1} \in \PFin{A} \right\}$.
      Let $C \in \Fin{C(a,i)}$. Then, we clearly have
      \[
      z \wedge \;\;
        \medwedge_{\mathclap{\substack{
        B_{0} \in \mathcal{V}^{*} \\
        B_{1} \in \Fin{A}}}}\;
        f(\sat \left(B_{0} \cup B_{1} \cup C
        \right)) \leq_{X} \widetilde{f}(\mathcal{V} \cup \left\{ 
          A \cup \left\{ b \right\} \mid b \in C \right\}).
      \]
      Then, by induction on the size of $C$, using assumption
      $\Phi(\left\{ A \cup \left\{ b \right\} \right\})$ for each
      $b \in C(a,i)$, one can easily show that
      $
      \widetilde{f}(\mathcal{V} \cup \left\{ A \cup \left\{ b
      \right\} \mid b \in C \right\}) \leq_{X}
      F(\mathcal{U} \cup \mathcal{V}).
      $
      Hence $\Phi(\left\{ A \cup \left\{ a \right\} \right\})$.
      Therefore $F$ is well defined.
      
      Now, it is straightforward to show that $F$ is a preframe
      homomorphism. 
      Finally, if $G \colon \Sat{\Lower{\mathcal{S}}} \to X$
      is any other preframe homomorphism such that $G \circ
      \sigma_{L}^{*} = f$, then for each $\mathcal{U} \subseteq
      \Fin{S}$, we have
      \begin{align*}
        G(\sat_{L} \mathcal{U})
        = \bigvee_{\mathclap{\mathcal{V} \in
        \Fin{\mathcal{U}}}} G(\sat_{L} \mathcal{V})
        &= \bigvee_{\mathcal{V} \in
        \Fin{\mathcal{U}}} \medwedge_{B \in
          \mathcal{V}^{*}}G(\medvee_{b \in B} \sat_{L} \left\{ b \right\})\\
        &= \bigvee_{\mathcal{V} \in
        \Fin{\mathcal{U}}} \medwedge_{B \in
          \mathcal{V}^{*}}G(\sigma_{L}^{*}(\sat B))\\
        &= \bigvee_{\mathcal{V} \in
        \Fin{\mathcal{U}}} \medwedge_{B \in
          \mathcal{V}^{*}}f(\sat B)
        = F(\sat_{L}\mathcal{U}).
        \qedhere
      \end{align*}
\end{proof}

If $\mathcal{S}$ is a continuous basic cover, then it is easy to see
that the dcpo homomorphisms from  $\Sat{\mathcal{S}}$ to $\Pow{\One}$
bijectively correspond to the Scott open subsets of
$\Sat{\mathcal{S}}$, and hence to the saturated subsets of
$\Scott(\mathcal{S})$. Similarly, if $\mathcal{S}$ is a stably compact
formal topology, then the preframe homomorphisms from
$\Sat{\mathcal{S}}$ to $\Pow{\One}$
bijectively correspond to the Scott open filters on
$\Sat{\mathcal{S}}$, and hence to the saturated subsets of the de
Groot dual $\mathcal{S}^{d}$ by Proposition~\ref{prop:deGrootSOF}.
Hence, by Proposition \ref{prop:preframeCoverage} with $X = \Pow{\One}$,
we conclude as follows.
\begin{theorem}
  The Scott topology $\Sigma(\mathcal{S})$ of a continuous basic cover
  $\mathcal{S}$ is the de Groot dual of the lower powerlocale
  $\Lower{\mathcal{S}}$.
\end{theorem}
\begin{corollary}
  For any continuous basic cover $\mathcal{S}$, we have
  $\Located(\mathcal{S}) \cong \Patch(\Scott(\mathcal{S}))$.
\end{corollary}
\begin{proof}
Immediate from Theorem \ref{thm:PatchdeGroot} and Theorem
\ref{thm:LawsonTop}.
\end{proof}

\subsection{The Vietoris monad}\label{sec:Vietoris}
By Theorem \ref{thm:LRegLFCov}, the construction of Lawson topologies
induces a monad on the category of compact regular formal topologies.
We show that this monad is the Vietoris monad
\cite[Chapter III, Section 4]{johnstone-82}.

We work over a fixed compact regular formal topology $\mathcal{S}$.
\begin{definition} \label{def:Vietoris}
The \emph{Vietoris powerlocale} of
$\mathcal{S}$ is a formal topology $\Viet(\mathcal{S})$ presented by the
geometric theory $T_{\Viet}$ over the
propositional symbols
\[
  P_{\Viet} \defeql \left\{ \sq A \mid A \in \Fin{S} \right\}
  \cup \left\{ \dia a  \mid a \in S \right\}
\] 
with the following axioms:
\begin{enumerate}
  \myitem[($\dia$1)]\label{d1}  $\dia a  \vdash \bigvee \left\{ 
  \dia b \mid b \in A \right\} \qquad (a \cov A \in \Fin{S})$
  
  \myitem[($\dia$2)]\label{d2} $\dia a  \vdash
  \bigvee\left\{ \dia b \mid b  \ll a \right\}$ 

  \myitem[($\sq$1)]\label{s1} $ \top  \vdash \bigvee \left\{ 
    \sq A \mid  A \in \Fin{S} \right\}$
  
  \myitem[($\sq$2)]\label{s2} $\sq A \vdash \sq B \qquad (A \cov B)$
  
  \myitem[($\sq$3)]\label{s3} $\sq A \wedge \sq B
        \vdash \bigvee \left\{ \sq C \mid C \ll A \amp C \ll B \right\} $
  
  \myitem[(V1)]\label{v1} $\sq A \wedge \dia a
        \vdash \bigvee \left\{ \dia b \mid b \in A \downarrow a \right\} $

  \myitem[(V2)]\label{v2}  $\sq(A \cup \left\{ a \right\})
        \vdash \sq A \vee \dia a$
\end{enumerate}
\end{definition}
\begin{remark}
The axioms \ref{d1} and \ref{d2} are essentially the same as
the axioms \ref{r3} and \ref{r4} of the theory $T_{\Located}$.
Moreover,
the axioms \ref{s1} -- \ref{s3} are analogous to the axioms \ref{l1} -- 
\ref{l3} of $T_{\Located}$ except that the relations $\cov$ and $\ll$
are reversed. Although we do not elaborate,
there is an important reason for this: the geometric theory over 
$\left\{ \sq A \mid A \in \Fin{S} \right\}$ with the axioms 
\ref{s1} -- \ref{s3} presents the upper powerlocale of $\mathcal{S}$
which is the Lawson dual of the Scott topology $\Scott(\mathcal{S})$
in the sense of information system for continuous posets; see
\citet[Theorem 3.5]{Infosys}.
\end{remark}

The theory $T_{\Viet}$ also describes the space of located subsets of
$\mathcal{S}$ in the sense of Proposition \ref{prop:PtV} below.
This observation is due to \citet[Proposition 73]{Spitters10LocatedOvert}.
Note that in compact regular formal topologies, closed subtopologies
and compact subtopologies are equivalent \cite[Chapter III,
Proposition 1.2]{johnstone-82}.
\begin{proposition}\label{prop:PtV}
  If $m \subseteq P_{\Viet}$ is a model of $T_{\Viet}$, then
  \[
    U_{m} \defeql \left\{ a \in S \mid \dia a \in m \right\}, \qquad
    \mathcal{C}_{m} \defeql \left\{ A  \in \Fin{S} \mid \sq A \in m \right\} 
  \]
  are respectively a located subset of $\mathcal{S}$ and 
  the set of finite covers of the closed subtopology
  of $\mathcal{S}$ determined by the complement $\neg U_{m} \defeql
  \left\{ a \in S \mid a \notin U_{m} \right\}$.

  Conversely, if $U \subseteq S$ is a located subset of
  $\mathcal{S}$, then
  \[
    m_{U} \defeql \left\{ \dia a \mid a \in U \right\} \cup \left\{ \sq A \mid S
      \cov' A \right\}
  \]
  is a model of $T_{\Viet}$, where $\cov'$ is the cover of the closed
  subtopology of $\mathcal{S}$ determined by $\neg U$.
  Moreover, the correspondence is bijective.
\end{proposition}
\begin{proof}
  First, note that in a compact regular formal topology
  $\mathcal{S}$, we have 
\begin{equation}\label{eq:Normalty}
  U \ll V \iff\left( \exists C \in \Fin{S}\right) C \downarrow U \cov
  \emptyset \amp S \cov C \cup V.
\end{equation}

\noindent($\Rightarrow$) Let $m \subseteq P_{\Viet}$ be a model of
$T_{\Viet}$. First, we show that $U_{m}$ is a located subset.
  Note that $U_{m}$ is splitting by \ref{d1} and
  \ref{d2}. Suppose that $a \ll b$. 
  By \eqref{eq:Normalty}, there exists $A \in \Fin{S}$ such that 
  $A \downarrow a \cov \emptyset$ and $S \cov A \cup \left\{ b
  \right\}$.
  Thus, $\sq(A \cup \left\{ b \right\}) \in m$ by \ref{s1} and \ref{s2}.
  Either $\sq A \in m$ or  $\dia b \in m$ by \ref{v2}. In the
  former case, if $a \in U_{m}$, then there exists $c \in A \downarrow
  a$ such that $\dia c \in m$ by \ref{v1}, which is a contradiction by
  \ref{d1}. Thus $a \notin U_{m}$. In the latter case, we have
  $b \in U_{m}$. Hence $U_{m}$ is located.
  Next, let $\cov'$ denote the cover of the closed
  subtopology of $\mathcal{S}$ determined by $\neg U_{m}$, and let $A \in
  \mathcal{C}_{m}$.  We must show that $S \cov' A$, i.e.\ $S \cov \neg
  U_{m} \cup A$.  By \ref{s3}, there exists $B \ll A$ such that $\sq B
  \in m$. Then $S \cov B^{*} \cup A$, and so 
  $B^{*} \subseteq \neg U_{m}$ by \ref{v1}  and \ref{d1}. Hence $S \cov' A$.
  Conversely, suppose that  $S \cov' A$. Then there exists $\sq (B_{0} \cup
  B_{1}) \in m$ such that $B_{0} \subseteq \neg U_{m}$ and
  $B_{1} \subseteq A$ by \ref{s1} and \ref{s2}. Then, $\sq
  B_{1} \in m$ by repeated applications of \ref{v2}, and so $\sq A \in m$ by \ref{s2}.

  \medskip
  \noindent($\Leftarrow$) Let $U \subseteq S$ be a located subset of
  $\mathcal{S}$. We show that $m_{U}$ satisfies \ref{v2}. Other
  axioms are straightforward to check. Suppose that $\sq(A \cup
  \left\{ a \right\}) \in m_{U}$. Then, $S \cov' A \cup \left\{ a
  \right\}$. Since $\mathcal{S}$ is compact regular, there exists $B \ll a$
  such that $S \cov' A \cup B$. Since $U$ is located, either $B
  \subseteq \neg U$ or $a \in U$. In the former case, we have
  $S \cov' A$, and so $\sq A \in m_{U}$. In the latter
  case, we have $\dia a \in m_{U}$.

  The above correspondence is clearly bijective.
\end{proof}

By the observation made in Proposition \ref{prop:PtV}, we can define a
translation between the models of $T_{\Located}$ and
the models of $T_{\Viet}$ as follows:
\begin{align*}
  \begin{aligned}
  m \in \Model{ T_{\Located} } &\mapsto \left\{ \dia a \mid
  \rcut(a) \in m \right\} \cup \left\{ \sq A \mid \left( \exists\,
    \lcut(B) \in m \right) S \cov B \cup A \right\}, \\
  M \in \Model{ T_{\Viet} } &\mapsto \left\{ \rcut(a) \mid
  \dia a \in M \right\} \cup \left\{ \lcut(A) \mid \left( \exists\,
    \sq B \in M \right)  B  \downarrow  A \cov \emptyset \right\}.
  \end{aligned}
\end{align*}
It is straightforward to show that these mappings are well-defined and bijective.
Motivated by this correspondence, we define formal
topology maps $r \colon \Located(\mathcal{S}) \to \Viet(\mathcal{S})$
and $s \colon \Viet(\mathcal{S}) \to \Located(\mathcal{S})$ by
specifying their actions on generators as follows:
\begin{align}
  &\begin{aligned}
  \elPT{A} \mathrel{r} \left\{ \dia a \right\}
  &\defeqiv \elPT{A} \cov_{\Located(\mathcal{S})} \left\{ \rcut(a) \right\},
  \\
  \elPT{A} \mathrel{r} \left\{ \sq A \right\}
  &\defeqiv \elPT{A} \cov_{\Located(\mathcal{S})} \left\{ \left\{ \lcut(B)\right\} \mid S \cov A
  \cup B \right\}, \label{def:NatIsor}
  \end{aligned} \\
  &\begin{aligned}
  \elPT{A} \mathrel{s} \left\{ \rcut(a) \right\}
  &\defeqiv \elPT{A} \cov_{\Viet(\mathcal{S})} \left\{ \dia a \right\},
  \\
  \elPT{B} \mathrel{s} \left\{ \lcut(A) \right\}
  &\defeqiv \elPT{B} \cov_{\Viet(\mathcal{S})} \left\{ \left\{ \sq B \right\} \mid B \downarrow A
    \cov \emptyset \right\}, \label{def:NatIsos}
  \end{aligned}
\end{align}

\begin{lemma} \label{lem:VietMaps}
  $r$ and $s$ respect the axioms of $T_{\Viet}$ and $T_{\Located}$ respectively.
\end{lemma}
\begin{proof}
  1.\ $r$ respects the axioms of $T_{\Viet}$:
  We show that $r$ respects \ref{v1} and \ref{v2}. Other axioms
  are straightforward to check.

  \medskip
  \noindent\ref{v1}
  Let $\elPT{A} \in r^{-} \left\{ \sq A \right\} \downarrow
  r^{-} \left\{ \dia a \right\}$. We may assume
  that $\elPT{A}$ is of the form $\left\{ \lcut(B), \rcut(a) \right\}$
  for some $B \in \Fin{S}$ such that $S \cov A \cup B$.  By \ref{r4}, we have
    $
    \elPT{A} \cov_{\Located(\mathcal{S})} \left\{ \left\{ \lcut(B), \rcut(b)
    \right\} \mid b \ll a \right\}.
    $
  Let $b \ll a$. Since $a \cov (a \downarrow A) \cup B$, there exists
  $C \in \Fin{a \downarrow A}$ such that $b \cov C \cup B$. Then by
  \ref{r3} and \ref{D}, we have
  $
  \left\{ \lcut(B), \rcut(b) \right\}
  \cov_{\Located(\mathcal{S})}
  \left\{ \left\{ \rcut(c) \right\} \mid c \in C \right\}.
  $
  Hence,
  $
  \elPT{A} \cov_{\Located(\mathcal{S})}
  r^{-} \left\{ \left\{ \dia c \right\}
  \mid c \in A \downarrow a \right\}.
  $

  \medskip
  \noindent\ref{v2}
  Let $\elPT{A} \in r^{-} \left\{ \sq \left(  A \cup \left\{ a
  \right\}\right) \right\}$. We may assume that $\elPT{A}$ is of the
  form $\left\{ \lcut(B) \right\}$ where $S \cov A \cup \left\{ a
  \right\} \cup B$. Since $\mathcal{S}$ is compact regular, there exists $C \in
  \Fin{S}$ such that $C \ll a$ and $S \cov C \cup A \cup B$. Then, 
  \begin{align*}
    \left\{ \lcut(B) \right\}
    &\cov_{\Located(\mathcal{S})}
    \left\{ \left\{ \lcut(B), \lcut(C) \right\} 
    \left\{ \lcut(B), \rcut(a) \right\} \right\} && \text{(by Lemma
      \ref{lem:LittleFact})} \\
    &\cov_{\Located(\mathcal{S})}
    \left\{ \left\{ \lcut(B \cup C) \right\} 
    \left\{ \lcut(B), \rcut(a) \right\} \right\} && \text{(by \ref{l2}
    and \ref{l3})} \\
    &\cov_{\Located(\mathcal{S})}
    \left\{ \left\{ \lcut(B \cup C) \right\} 
    \left\{ \rcut(a) \right\} \right\} \\
    &\cov_{\Located(\mathcal{S})}
    r^{-}\left\{ \left\{ \sq A \right\}, \left\{ \dia a
    \right\}\right\}.
  \end{align*}

  \noindent2.\ $s$ respects the axioms of $T_{\Located}$: We show that $s$
  respects \ref{Loc}. The other axioms can be treated similarly.
  Suppose that $a \ll b$.
  By \eqref{eq:Normalty}, there exists
  $C \in \Fin{S}$ such that $C \downarrow a \cov \emptyset$
  and $S \cov C \cup \left\{ b
  \right\}$. Then
  \begin{align*}
    \Fin{P_{\Viet}}
    &\cov_{\Viet(\mathcal{S})}
    \left\{ \sq (C \cup \left\{ b \right\}) \right\}
    && \text{(by \ref{s1} and \ref{s2})}\\
    &\cov_{\Viet(\mathcal{S})}
     \left\{ \left\{ \sq C \right\},
     \left\{ \dia b \right\} \right\}
    && \text{(by \ref{v2})}\\
    &\subseteq s^{-}\left\{ \left\{
      \lcut(\left\{ a \right\}) \right\}, \left\{ \rcut(b) \right\}
    \right\}. &&\qedhere
  \end{align*}
\end{proof}
Hence, $r$ and $s$ indeed uniquely extend to formal topology maps
$r \colon \Located(\mathcal{S}) \to \Viet(\mathcal{S})$ and 
$s \colon \Viet(\mathcal{S}) \to \Located(\mathcal{S})$.

\begin{lemma}
  $r$ and $s$ are inverse to each other.
\end{lemma}
\begin{proof}
  By the definition of $r$ and $s$, it suffices to show that 
  \begin{align*}
    \left\{ \lcut(A) \right\} &=_{\Located(\mathcal{S})}
    r^{-}s^{-}\left\{ \lcut(A) \right\}, &
    \left\{ \sq A \right\} &=_{\Viet(\mathcal{S})}
    s^{-}r^{-}\left\{ \sq A \right\}
  \end{align*}
  for each $A \in \Fin{S}$. These conditions are equivalent to
  \begin{align*}
    \left\{ \lcut(A) \right\} &=_{\Located(\mathcal{S})}
    \left\{ \left\{ \lcut(C) \right\} \mid \left( \exists B \in
    \Fin{S} \right) S \cov B \cup C \amp B \downarrow A \cov \emptyset
  \right\},\\
    \left\{ \sq A \right\} &=_{\Viet(\mathcal{S})}
    \left\{ \left\{ \sq C \right\} \mid \left( \exists B \in
    \Fin{S} \right) C \downarrow B \cov \emptyset \amp S \cov B \cup A
  \right\}.
  \end{align*}
  The former follows from \eqref{eq:Normalty} and \ref{l3};
  the latter follows from \eqref{eq:Normalty} and  \ref{s3}.
\end{proof}

The constructions $\Located(\mathcal{S})$ and $\Viet(\mathcal{S})$
extend to functors on $\KReg$. 
Their actions on a morphism $t \colon \mathcal{S}
\to \mathcal{S}'$ are defined as follows:
\begin{align*}
  \elPT{B} \mathrel{\Viet(t)} \left\{ \dia a \right\}
  &\defeqiv
  \elPT{B} \cov_{\Viet(\mathcal{S})} \left\{ \left\{ \dia b \right\} \mid b
  \mathrel{t} a \right\}, \\
  \elPT{B} \mathrel{\Viet(t)} \left\{ \sq A \right\}
  &\defeqiv
  \elPT{B} \cov_{\Viet(\mathcal{S})} \left\{ \left\{ \sq B \right\} \mid B \cov
  t^{-} A\right\},\\
  \elPT{A} \mathrel{\Located(t)} \left\{ \rcut(a) \right\}
  &\defeqiv
  \elPT{A} \cov_{\Located(\mathcal{S})} \left\{ \left\{ \rcut(b) \right\} \mid b
  \mathrel{t} a \right\}, \\
  \elPT{A} \mathrel{\Located(t)} \left\{ \lcut(A) \right\}
  &\defeqiv
  \elPT{A} \cov_{\Located(\mathcal{S})} \left\{ \elPT{A}' \mid 
  \left( \exists B \gg A \right) 
  \left\{ \elPT{A}' \cup \left\{ \rcut(a) \right\}\mid  a \in t^{-}B  \right\}
  \cov_{\Located(\mathcal{S})} \emptyset  \right\}.
\end{align*}
The  Vietoris functor $\Viet$ is a part of a monad
$\langle \Viet, \eta^{\Viet}, \mu^{\Viet} \rangle$ on $\KReg$; see
\citet[Chapter III, Section 4.5]{johnstone-82}. The structure of the
monad is defined as follows:
\begin{align*}
  a \mathrel{\eta^{\Viet}} \left\{ \dia b \right\}
  &\defeqiv
  a \cov b,\\
  a \mathrel{\eta^{\Viet}} \left\{ \sq B\right\}
  &\defeqiv
  a \cov B,\\
  \mathcal{U} \mathrel{\mu^{\Viet}} \left\{ \dia a \right\}
  &\defeqiv
  \mathcal{U} \cov_{\VViet(\mathcal{S})}
  \left\{ \dia \left\{ \dia a \right\} \right\}, \\
  \mathcal{U} \mathrel{\mu^{\Viet}} \left\{ \sq A\right\}
  &\defeqiv
  \mathcal{U} \cov_{\VViet(\mathcal{S})}
  \left\{ \sq \left\{ \left\{ \sq A \right\} \right\} \right\},
\end{align*}
where $\cov_{\VViet(\mathcal{S})}$ is the
cover of $\Viet(\Viet(\mathcal{S}))$ and
$\mathcal{U}$ ranges over the base of $\Viet(\Viet(\mathcal{S}))$.

By Theorem \ref{thm:LRegLFCov}, the functor $\Located$ also induces a
monad $\langle \Located, \eta^{\Located}, \mu^{\Located} \rangle$  on
$\KReg$:
\begin{align*}
  a \mathrel{\eta^{\Located}} \left\{ \rcut(b) \right\}
  &\defeqiv
  a \cov b, \\
  a \mathrel{\eta^{\Located}} \left\{ \lcut(B) \right\}
  &\defeqiv
  a \cov \left\{ a' \mid \left( \exists B' \gg B \right) a' \downarrow B' \cov \emptyset \right\},\\
  \mathcal{W} \mathrel{\mu^{\Located}} \left\{ \rcut(a)\right\}
  &\defeqiv
  \mathcal{W} \cov_{\LLocated(\mathcal{S})}
  \left\{ \rcut(\left\{ \rcut(a) \right\}) \right\},\\
  \mathcal{W} \mathrel{\mu^{\Located}} \left\{ \lcut(A) \right\}
  &\defeqiv
  \mathcal{W} \cov_{\LLocated(\mathcal{S})}
  \left\{ \mathcal{W}'\! \mid 
  \left( \exists B \gg\! A \right) 
  \left\{ \mathcal{W}' \cup \left\{\rcut(\left\{ \rcut(b) \right\})  \right\} \mid  b
\in B  \right\} \cov_{\LLocated(\mathcal{S})} \emptyset  \right\},
\end{align*}
where $\cov_{\LLocated(\mathcal{S})}$ is the
cover of $\Located(\Located(\mathcal{S}))$ and
$\mathcal{W}$ ranges over the base of $\Located(\Located(\mathcal{S}))$.

\begin{theorem}\label{thm:EquiLV}
  The morphism $r \colon \Located(\mathcal{S}) \to \Viet(\mathcal{S})$
  defined by \eqref{def:NatIsor} is an isomorphism between monads
  $\langle \Located, \eta^{\Located}, \mu^{\Located} \rangle$
  and $\langle \Viet, \eta^{\Viet}, \mu^{\Viet} \rangle$.
  Hence, the monad induced by the adjunction
    $
    U \dashv \Located \colon \KReg \to \LFCov
    $
  is the Vietoris monad on $\KReg$.
\end{theorem}
\begin{proof}
  We must show that $r$ is natural in $\mathcal{S}$, and it preserves 
  the unit $\eta$ and multiplication $\mu$. The latter means that
  the following diagrams commute:
  \[
    \xymatrix{
      & \mathcal{S} \ar[dl]_{\eta^{\Located}_{\mathcal{S}}}
      \ar[dr]^{\eta^{\Viet}_{\mathcal{S}}}&\\
      \Located(\mathcal{S}) \ar[rr]_{r_{\mathcal{S}}} & & \Viet(\mathcal{S})
    }
    \qquad
    \xymatrix@C+=3em{
      \Located(\Located(\mathcal{S}))
      \ar[d]_{\mu^{\Located}_{\mathcal{S}}}
      \ar[r]^{\Located(r_{\mathcal{S}})}
      & \Located(\Viet(\mathcal{S})) \ar[r]^{r_{\Viet(\mathcal{S})}}
      & \Viet(\Viet(\mathcal{S}))
      \ar[d]^{\mu^{\Viet}_{\mathcal{S}}}\\
      \Located(\mathcal{S}) \ar[rr]_{r_{\mathcal{S}}} & & \Viet(\mathcal{S})
    }
  \]
  Here, by writing $r_{\mathcal{S}}$, we consider $r \colon
  \Located(\mathcal{S}) \to \Viet(\mathcal{S})$ to be parameterised by
  $\mathcal{S}$.
  Let $t \colon \mathcal{S} \to \mathcal{S}'$ be a morphism
  between compact regular formal topologies.  We must show that
  $r_{\mathcal{S}'} \circ \Located(t) = \Viet(t) \circ
  r_{\mathcal{S}}$.  
  By  Lemma \ref{lem:RegMax} and by the definitions of $\Located$ and
  $\Viet$, it suffices to show that
  \[
    \left( r_{\mathcal{S}'} \circ \Located(t)\right)^{-}\left\{ \sq B \right\}
    \cov_{\Located(\mathcal{S})} (\Viet(t) \circ r_{\mathcal{S}})^{-}\left\{ \sq B \right\}
  \]
  for each $B \in \Fin{S'}$.
  Let $\elPT{A} \in \left( r_{\mathcal{S}'} \circ
  \Located(t)\right)^{-}\left\{ \sq B \right\}$. We may assume that
  there exist $A,A' \in \Fin{S'}$ such that $S' \cov'
  A \cup B$, $A \ll A'$  and $\elPT{A} \cup  \left\{ \rcut(a') \right\}
  \cov_{\Located(\mathcal{S})} \emptyset$ for each $a' \in
  t^{-}A'$.
  Then, $S \cov t^{-}A' \cup t^{-}B$. Since $\mathcal{S}$ is compact regular,
  there exist  $C,C' \in
  \Fin{S}$ such that $C \ll C' \subseteq t^{-}A'$ and
   $S \cov C \cup t^{-}B$. Then, by Lemma~\ref{lem:LittleFact}, we
   have
  \begin{align*}
  \elPT{A}
  &\cov_{\Located(\mathcal{S})}
  \left\{ \elPT{A} \cup \left\{ \lcut(C) \right\} \right\}
  \cup \left\{ \elPT{A} \cup \left\{ \rcut(c') \right\} \mid c' \in C' \right\}\\ 
  &\cov_{\Located(\mathcal{S})}
  \left\{ \left\{ \lcut(C) \right\} \right\}
  \subseteq 
    \left(\Viet(t) \circ r_{\mathcal{S}}\right)^{-}\left\{ \sq B
    \right\}.
\end{align*}
Hence, $r_{\mathcal{S}}$ is a natural isomorphism between $\Located$ and
$\Viet$.

The commutativity of the left diagram is straightforward.
For the diagram on the right, it suffices to show that
  \[
    \left( r_{\mathcal{S}} \circ \mu^{\Located}_{\mathcal{S}} \right)^{-}\left\{ \sq B \right\}
    \cov_{\LLocated(\mathcal{S})}
    (\mu^{\Viet}_{\mathcal{S}} \circ
    r_{\Viet(\mathcal{S})} \circ \Located(r_{\mathcal{S}}))^{-}\left\{ \sq B \right\}
  \]
  for each $B \in \Fin{S}$.  Suppose that
  $\mathcal{U} \in
  \left( r_{\mathcal{S}} \circ
  \mu^{\Located}_{\mathcal{S}} \right)^{-}\left\{ \sq B \right\}$.
  We may assume that
  there exist $A, A' \in \Fin{S}$ such that
  $S \cov A \cup B$, $A' \gg A$, and  $\mathcal{U} \cup \left\{ \rcut(\left\{
  \rcut(a') \right\}) \right\} \cov_{\LLocated(\mathcal{S})} \emptyset$
  for each $a' \in A'$. Then
  \begin{align*}
    \Fin{P_{\Viet}}
    &\cov_{\Viet(\mathcal{S})}
    \left\{ \sq(A \cup B)\right\} && \text{(by \ref{s1} and \ref{s2})}\\
    &\cov_{\Viet(\mathcal{S})}
    \left\{ \left\{ \dia a  \right\} \mid a \in A \right\} \cup \left\{
      \left\{ \sq B\right\}\right\} && \text{(by \ref{v2})}.
    \end{align*}
  Put $\mathcal{A} = \left\{ \left\{ \dia a \right\} \mid a \in A \right\}$
  and 
  $\mathcal{A}' = \left\{ \left\{ \dia a' \right\} \mid a' \in A'
  \right\}$. Then, $\left\{ \lcut(\mathcal{A}) \right\}
  \mathrel{r_{\Viet(\mathrel{S})}} \left\{ \sq \left\{\left\{ \sq B
  \right\}  \right\} \right\}$, and since $\mathcal{A} \ll
  \mathcal{A}'$ in $\Viet(\mathcal{S})$, we have $\mathcal{U}
  \mathrel{\Located(r_{\mathcal{S}})} \left\{ \lcut(\mathcal{A})
  \right\}$. Hence,
  $\mathcal{U}
    \cov_{\LLocated(\mathcal{S})}
    (\mu^{\Viet}_{\mathcal{S}} \circ
    r_{\Viet(\mathcal{S})} \circ \Located(r_{\mathcal{S}}))^{-}\left\{ \sq B \right\}
  $.
\end{proof}

Theorem \ref{thm:EquiLV} is a point-free version of the fact that
the category of algebras of the Vietoris monad is equivalent to that of
Lawson semilattices and Lawson continuous homomorphisms;
see \citet[Chapter VII, Theorem 3.6]{johnstone-82}.
Classically, the functor $\Located \colon \LFCov \to \KReg$ must
be monadic, which demonstrates the fact that the category of continuous
lattices and perfect maps are equivalent to that of Lawson
semilattices and Lawson continuous homomorphisms.
We could not find a constructive version of this classical fact.

\section*{Acknowledgements}
The author was supported by Istituto Nazionale di Alta Matematica ``F.
Severi'' (INdAM) as a fellow of INdAM-COFUND-2012.

\bibliographystyle{myabbrvnat}

\end{document}